\documentclass[11pt]{article}
 \usepackage[margin=1in]{geometry}
\usepackage{float}
\usepackage{amsmath,amsfonts,mathrsfs,wasysym,times,dsfont}
\usepackage{amsthm,amscd,amsfonts,newlfont}
\usepackage{amssymb, upref, color}
\usepackage[dvips]{graphicx}
\usepackage{epsf, subfigure, verbatim}
\usepackage{latexsym}
\usepackage[colorlinks]{hyperref}
\usepackage{mathrsfs}
\usepackage{multirow}
  \usepackage{lipsum,multicol}
\usepackage{color}
\usepackage{setspace}
\numberwithin{equation}{section}

\theoremstyle{plain}
\newtheorem{exam}{Example}[section]
\newtheorem{theorem}[exam]{Theorem}
\newtheorem{lemma}[exam]{Lemma}
\newtheorem{remark}[exam]{Remark}

\newtheorem{proposition}[exam]{Proposition}

\newtheorem{corollary}[exam]{Corollary}

\title{Multiple Limit Cycles and Heteroclinic Loops in a Predator-prey System with Allee Effects in Prey}
\author{Alessandro Arsie\footnotemark[2], \ Chanaka Kottegoda\footnotemark[2], \ and \  Chunhua Shan\footnotemark[2]\ \ \footnotemark[4]\ \ \footnotemark[1]}

\date{}

\begin{document}
\maketitle

\begin{abstract}

The transition between strong and weak Allee effects in prey provides a simple regime shift in ecology. A deteriorating environment changes weak Allee effects into strong ones. In this paper, we study the interplay between the functional response of Holling type IV and both strong and weak Allee effects. The model investigated here presents complex dynamics and high codimension bifurcations. In particular, nilpotent cusp bifurcation, nilpotent saddle bifurcation and degenerate Hopf bifurcation of codimension 3 are completely analyzed, and the existence of homoclinic and heteroclinic loops are proven. Remarkably it is the first time that three limit cycles are discovered in predator-prey models with Allee effects. It turns out that strong Allee effects destabilize population dynamics, induce more regime shifts, decrease establishment likelihood of both species, increase vulnerability of ecosystem to collapse, while weak Allee effects promote sustained oscillations between predators and preys compared to systems without Allee effects. The theory developed here provides a sound foundation for understanding predator-prey interactions and biodiversity of species in natural systems.



\end{abstract}


\footnotetext[2]{Department of Mathematics and Statistics, The University of Toledo, Toledo, OH 43606, USA.}
\footnotetext[4]{Research of C. Shan was partially supported by Simons Foundation-Collaboration Grants for Mathematicians 523360.}
\footnotetext[1]{Corresponding author. {\it E-mail address:} chunhua.shan@utoledo.edu (C. Shan).}

\section{Introduction}\label{introduction}

Predator-prey models have been extensively studied by biologists and mathematicians to understand the population dynamics of two or more species among which predation occurs. In the literature, various functions have been used to model the predator response. However, the prey population is generally assumed to follow the logistic growth in the absence of predators. In contrast, abundant observational data have shown that Allee effect or depensation occurs in small or sparse populations, where species population is reduced at low density \cite{allee, allee1, Davis, Johnson, allee2, allee4, allee3}. In this paper we consider a predator-prey model with Holling type IV functional response and Allee effects in preys, which is given by
\begin{equation}  \label{main}
\left\{\begin{array}{l}
\displaystyle \frac{dx}{dt}=\displaystyle rx(K-x)(x-A)-p(x)y,\\
\displaystyle \frac{dy}{dt}=\displaystyle y(-d+cp(x)),\\
 x(0)\geq 0, \ y(0)\geq 0,
\end{array}\right.
\end{equation}
where $x$ and $y$ represent the densities of prey and predator populations, respectively, and $$p(x):=\frac{mx}{ax^2+bx+1}, \ \ a>0, \ b>-2\sqrt{a}, \ m>0,$$ is the generalized Holling type IV functional response \cite{Andrews, ZCW}.

In this model, $r$ denotes the intrinsic growth rate of the prey population and $K$ denotes the environmental carrying capacity for the prey. The constant of proportionality is denoted by $c$ and the natural death rate of the predators is represented by $d$. The parameters $r, K, c$ and $d$ are assumed to be positive. The response function $p(x)$ of Holling type IV is also known as Monod-Haldane response function \cite{Andrews}. It is positive when $x>0$, and increasing for $x$ between $0$ and $\frac{1}{\sqrt{a}}$. The function attains the maximum value at $x=\frac{1}{\sqrt{a}}$, decreases for $x>\frac{1}{\sqrt{a}}$ and approaches to zero as $x$ tends to infinity. This response function describes the phenomenon that preys can better disguise, protect and defend themselves against predators when their population becomes large enough ~\cite{FW, group1}.

The parameter $A$ represents the threshold of Allee effect or depensation in preys, which was first discovered by W. C. Allee in 1931 who observed and demonstrated the connection between the growth rate and the number of individuals of goldfish in a tank ~\cite{allee, allee1}. In general, Allee effects describe the relationship between population density and individual fitness (per capita population growth rate). Such effects could be either strong or weak  for a given species ~\cite{allee2, allee4, allee3}. A strong Allee effect means that the per capita growth rate is negative when the density is zero (or in the limit as the density goes to zero), which indicates that there is a critical value in population size, below which the population will tend towards extinction and above which the population grows to its carrying capacity. This happens when $0<A<K$ in model \eqref{main}. The weak Allee effect is used to describe the case where the population growth rate is negatively affected by low population size, but the per capita population growth rate cannot go below zero (so population still grows at low population size), and this happens when $-K<A<0$. Note that $A=0$ is the transition between the weak and strong Allee effects.

If the term $x-A$ in model \eqref{main} is neglected, then the prey population follows the logistic growth rate and Allee effects are not present. For predator-prey system with logistic growth rate and Holling type IV functional response, Freedman and Wolkowicz proved the existence of homoclinic bifurcation, Hopf bifurcation and discussed the possibility of the existence of two limit cycles in 1988 ~\cite{wol}. In 2001, Ruan and Xiao considered the case $b=0$ in the response function and studied the Bogdanov-Takens bifurcation of codimension 2. They also found the set of parameters for which the system has no limit cycles or a unique limit cycle ~\cite{SD}. Later on, Zhu, Campbell and Wolkowicz showed that the system has rich dynamics such as Bogdanov-Takens bifurcation of codimension 3. Furthermore,  they studied the degenerate Hopf bifurcation and proved that its codimension is at least 2 ~\cite{ZCW}. In 2006, Xiao and Zhu analyzed the Hopf bifurcation and showed that the codimension of the degenerate Hopf bifurcation is exactly 2 ~\cite{XZ}. Besides that, various response functions have been introduced in  predator-prey systems with logistic growth rate in prey population \cite{type31, huang, LCR, X}. In recent years, there has been an increased activity in studying Allee effects, largely because of their potential role in explaining the extinctions of already endangered, rare or dramatically declining species. For instance, the impact of Allee effects on invasive species has been explored by Lewis \cite{Lewis}, Hastings \cite{Taylor}, Wang  \cite{allee3}. However, much less attention has been paid to Allee effects on predator-prey systems with Holling type functional response. In 2011, Wang, Shi and Wei analyzed a class of predator-prey models with a strong Allee effect in a prey population, and showed that there is a unique limit cycle for the Holling type II functional response  \cite{JJJ}.  In 2013, Lin, Liu and Lai analyzed the predator-prey model with weak Allee effect using the response function $p(x)=\frac{mx}{a+x^2}$, where partial results regarding limit cycles were obtained \cite{RSX}. 

In this paper we study the interplay between the generalized Holling type IV functional response and both strong and weak Allee effects. The model exhibits  complex dynamics and high codimension bifurcations which are thoroughly investigated. Nilpotent cusp bifurcation of codimension 3, nilpotent saddle bifurcation and degenerate Hopf bifurcation of codimension 3 are completely and rigorously analyzed. To the best of our knowledge, it is the first time that a degenerate Hopf bifurcation of codimension 3 is discovered and studied in predator-prey models. Singular orbits including homoclinic orbits and heteroclinic orbits are analyzed. It turns out that strong Allee effects will destabilize population dynamics, induce more regime shifts, decrease establishment likelihood of both species, increase vulnerability of the ecosystem to collapse, while weak Allee effects promote sustained oscillations between predators and preys compared to systems without Allee effects.

The paper is organized as follows. In section 2, we examine the existence and linear stability of the equilibria of the system \eqref{third}. Section 3 is devoted to the discussion of the cusp singularity of order 3 and to the development of a universal unfolding of the Bogdanov-Takens bifurcation of codimension 3. By converting the system \eqref{third} to a generalized Li\'enard system, we study the degenerate Hopf bifurcation of codimension 3 in section 4. The existence and the non-existence of periodic orbits are also discussed. In section 5 we analyze the nilpotent saddle bifurcation and explore the distinct dynamics induced by the strong Allee effects. The impact of both weak and strong Allee effects on the interaction of predators and preys is explored. Biological interpretations of the significance of Allee effects are also provided.

System \eqref{main} has 8 parameters and using the scaling
\begin{equation*}
    (t,x,y)\rightarrow{}\left(\frac{m^2c^2}{r}t,\frac{r}{mc}x,\frac{r}{mc^2}y\right)
\end{equation*}
we can eliminate parameters $r, m$ and $c$. Hence, the system we consider is
\begin{equation}\label{third}
\left\{\begin{array}{rcl}
\displaystyle \frac{dx}{dt}&=& \displaystyle x(K-x)(x-A)-p(x)y=p(x)(G(x)-y),\\
\displaystyle \frac{dy}{dt}&=&\displaystyle y(-d+p(x)),
\end{array}\right.
\end{equation}
where
\begin{equation*}
  p(x):=\frac{x}{ax^2+bx+1} \ \ \ \textrm{and} \ \ \  G(x):=(x-A)(K-x)(ax^2+bx+1),
\end{equation*}
with the parameters
\begin{equation*}
    (K, A, a, b, d)\rightarrow\left(\frac{mc}{r}K, \frac{mc}{r}A, \frac{m^2c^2}{r^2}a, \frac{mc}{r}b, \frac{r}{m^2c^2}d\right).
\end{equation*}
Here, parameters $K,a,d$ are positive, $b>-2\sqrt[]{a}$ and $-K<A<K$.

The $x$-axis, $y$-axis and the nonnegative cone are invariant with respect to the flow of the system \eqref{third}.  It is clear that $p(x)$ and $G(x)$ are bounded from above for $x\geq 0$ and so is the product $p(x)G(x)$. Therefore, there exists a constant $M>0$ such that for any point $(x,y)$ in the set $\{(x,y)|x+y=N, N\geq M\}$, we have $$p(x)G(x)<dy \ \ \text{and}\ \  \dfrac{dx}{dt}+\dfrac{dy}{dt}=p(x)G(x)-dy<0.$$ So all orbits of system \eqref{third} eventually approach, enter and stay in the compact set enclosed by $x$-axis, $y$-axis and the line $x+y=M$. Hence, the system \eqref{third} is biologically meaningful.


\section{Linear stability analysis}

\subsection{Equilibrium points}

System \eqref{third} always has two equilibrium points on the nonnegative $x$-axis: $E_0 = (0, 0)$, representing the extinction of both species and $E_K = (K, 0)$, representing the extinction of predator population. There is a third equilibrium $E_A=(A,0)$, representing the threshold of Allee effects. When $A$ is negative, $E_A$ does not exist. As $A$ varies from negative to positive, $E_A$ coalesces with $E_0$ at the origin (when $A=0$) and moves right along the positive $x$-axis.

We observe that $p(x)=d$ if and only if $h(x)=0$ where
\begin{equation}\label{hx}
    h(x)=adx^2+(bd-1)x+d=0.
\end{equation}
Let $\alpha \leq\beta$ be the two possible positive roots, where
\begin{equation} \label{roots}
    \alpha = \frac{1-bd-\sqrt{(bd-1)^2-4ad^2}}{2ad} \ \ \textrm{and}\ \  \beta = \frac{1-bd+\sqrt{(bd-1)^2-4ad^2}}{2ad}.
\end{equation}
Correspondingly, we have two possible equilibria $E_{\alpha}=(\alpha,G(\alpha))$ and $E_{\beta}=(\beta,G(\beta))$.

Let $d_m=\frac{1}{b+2\sqrt{a}}$. Indeed $\alpha$ and $\beta$ exist when $d\in (0, d_m)$.  As $d$ increases, $\alpha$ increases, $\beta$ decreases, $\alpha$ and $\beta$ coalesce at $x=\frac{1}{\sqrt{a}}$ when $d=d_m$. $\alpha$ and $\beta$ are complex when $d>d_m$. For biological meaning we only consider equilibrium points in the positive cone. From the graph of $G(x)$, one can see that $E_{\alpha}$ and/or $E_{\beta}$ are in the positive cone only if $\alpha$ and/or $\beta$ are in the open interval $(\max\{0, A\}, K)$. We sketch regions in which equilibrium points exist in Fig. \ref{two_figure}.


\subsection{Linear analysis}

By using variational matrix, we investigate stability of equilibrium points $E_0,\ E_A, \ E_K, \  E_\alpha$ and $E_\beta$. The variational matrix of system \eqref{third} at an equilibrium, say, $(x^*, y^*)$ is given by
\begin{equation} \label{vm}
\mathbb{V}(x^*,y^*)=
\begin{bmatrix}
p'(x^*)(G(x^*)-y^*)+p(x^*)G'(x^*) & -p(x^*) \\
p'(x^*)y^*  &  p(x^*)-d
\end{bmatrix}.
\end{equation}
If the equilibrium $(x^*, y^*)$ is on the $x$-axis, i.e., $x^*\in\{0, A, K\}$ and $y^*=0$,  $\mathbb{V}(x^*,y^*)$ has two eigenvalues $\lambda_1=p(x^*)G'(x^*)$ and $\lambda_2=p(x^*)-d$. For the equilibrium $(x^*, y^*)$ in the interior of the positive cone, i.e., $x^*\in\{\alpha, \beta\}$ and $y^*=G(x^*)$, we have
\begin{equation}\label{td}
\left.\begin{split}
   tr(\mathbb{V}(x^*, G(x^*)))&=p(x^*)G'(x^*),    \\
   det(\mathbb{V}(x^*, G(x^*)))&=p(x^*)p'(x^*)G(x^*)=\displaystyle \frac{x^*(1-a(x^*)^2)(K-x^*)(x^*-A)}{(a(x^*)^2+bx^*+1)^2}.
  \end{split}
  \right.
\end{equation}

Standard linear stability analysis and straightforward calculation lead to the results summarized in the two tables below, in which ``DNE'' means ``Does Not Exist'', ``attr.'' and ``repe.'' denote ``attracting'' and ``repelling'', respectively.

\begin{table}[H]
\begin{center}
\begin{tabular}{|c|c|c|c|c|c|c|}
\hline
Conditions & region & $E_0$ & $E_A$ & $E_K$ & $E_\alpha$ & $E_\beta$ \\ \hline
$\alpha<A<K<\beta$ & $V_0^{(1)}$ & stable node & unstable node & saddle  & DNE & DNE\\ \hline
$\alpha<\beta< A<K$ & $V_0^{(2)}$ & stable node & saddle & stable node & DNE & DNE \\ \hline
$A<K<\alpha<\beta$ & $V_0^{(3)}$ & stable node  & saddle & stable node & DNE & DNE\\ \hline
$\alpha<A<\beta<K$ & $V_\beta$ & stable node & unstable node & stable node &DNE  & saddle \\ \hline
$A<\alpha<K<\beta$& $V_\alpha$ & stable node  & saddle & saddle & attr. $(G'(\alpha)<0)$ & DNE \\
                   & &             &        &             & repe. $(G'(\alpha)>0)$ &      \\ \hline
$A<\alpha<\beta<K$& $V_{\alpha\beta}$ & stable node & saddle & stable node & attr. $(G'(\alpha)<0)$ & saddle\\
                  &  &             &        &             & repe. $(G'(\alpha)>0)$ &      \\ \hline
\end{tabular}\caption{Existence and stability of equilibria with strong Allee effect ($0<A<K$).}
\end{center}
\end{table}

From the stability analysis, we can see that a transcritical bifurcation involving $E_0$ and $E_A$ occurs when $A$ change from negative to positive. With strong Allee effects $E_0$ is a stable node, and preys will die out at low population size. While with weak Allee effects $E_0$ is a saddle, and the unstable manifold of $E_0$ is on the $x$-axis, so prey population still grows at low size. The stability of $E_{\alpha}$ depends on the sign of $G'(\alpha)$. The equilibrium point $E_{\beta}$, whenever it exists, is always a saddle point. 

\begin{table}[H]
\begin{center}
\begin{tabular}{|c|c|c|c|c|c|}
\hline
Conditions & Region &$E_0$ &  $E_K$ & $E_\alpha$ & $E_\beta$ \\ \hline
$K<\alpha<\beta$ & $V_0^{(3)}$ &  saddle & stable node   & DNE & DNE\\ \hline
$\alpha<K<\beta$ & $V_{\alpha}$ & saddle &  saddle & attr. $(G'(\alpha)<0)$ & DNE\\
                  &  &                   &             & repe. $(G'(\alpha)>0)$ &      \\ \hline
$\alpha<\beta<K$ & $V_{\alpha\beta}$ & saddle & stable node   & attr. $(G'(\alpha)<0)$ & saddle \\
                  &  &                   &             & repe. $(G'(\alpha)>0)$ &      \\ \hline
\end{tabular}\caption{Existence and stability of equilibra with weak Allee effect ($-K<A<0$).}
\end{center}
\end{table}

\subsection{Saddle-node and transcritical bifurcations}

Let $(d, A)\in (0, \infty)\times (-K,K)$. As $d$ changes, we obtain the following bifurcations. See Fig. \ref{two_figure}.

\begin{figure}[!htp]
\begin{center}
    \subfigure[{$K>\frac{1}{\sqrt{a}}$}]
    {\includegraphics[angle=0,width=0.36\textwidth]{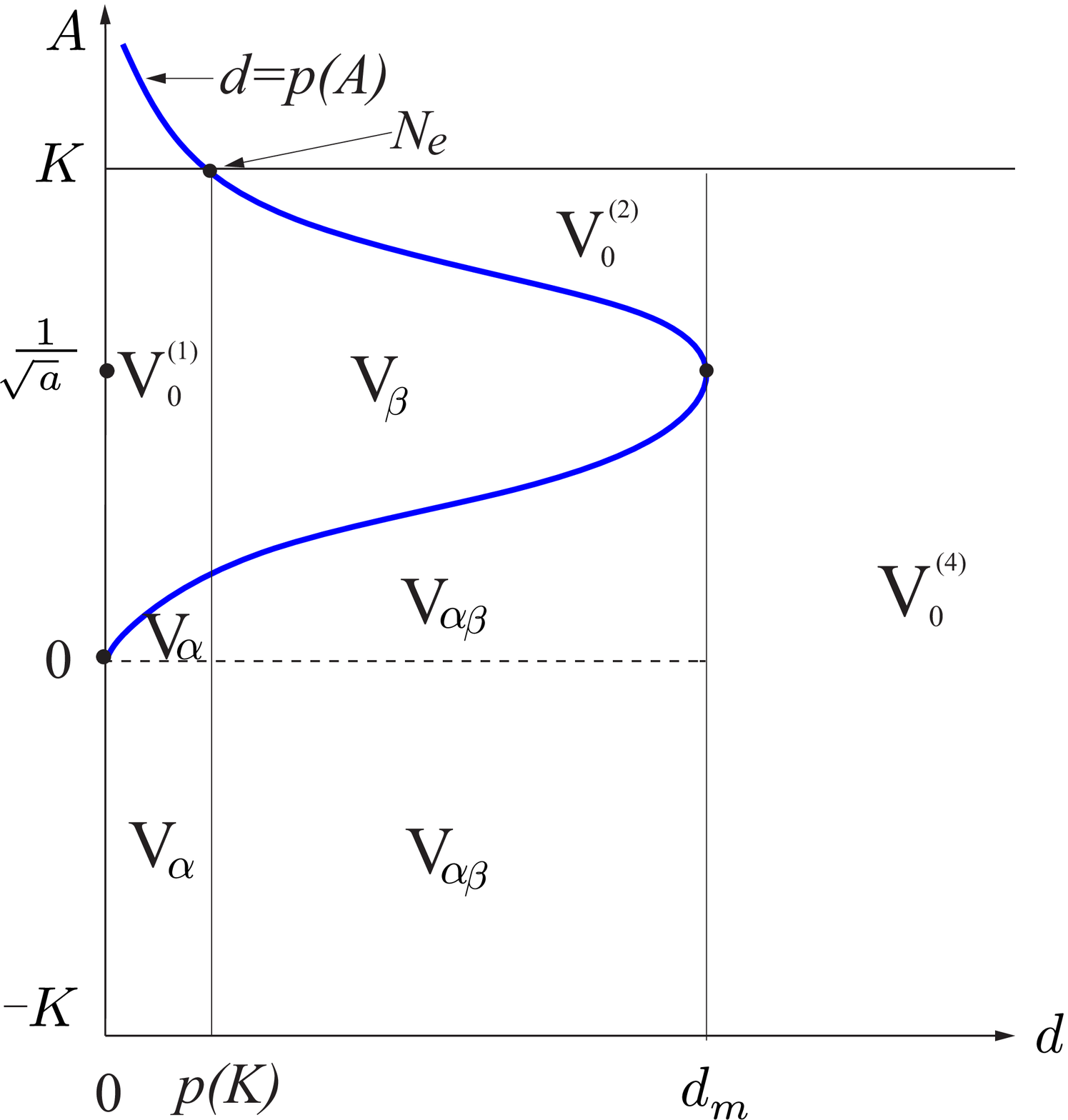}}\qquad \quad
\subfigure[{$K<\frac{1}{\sqrt{a}}$}]
    {\includegraphics[angle=0,width=0.36\textwidth]{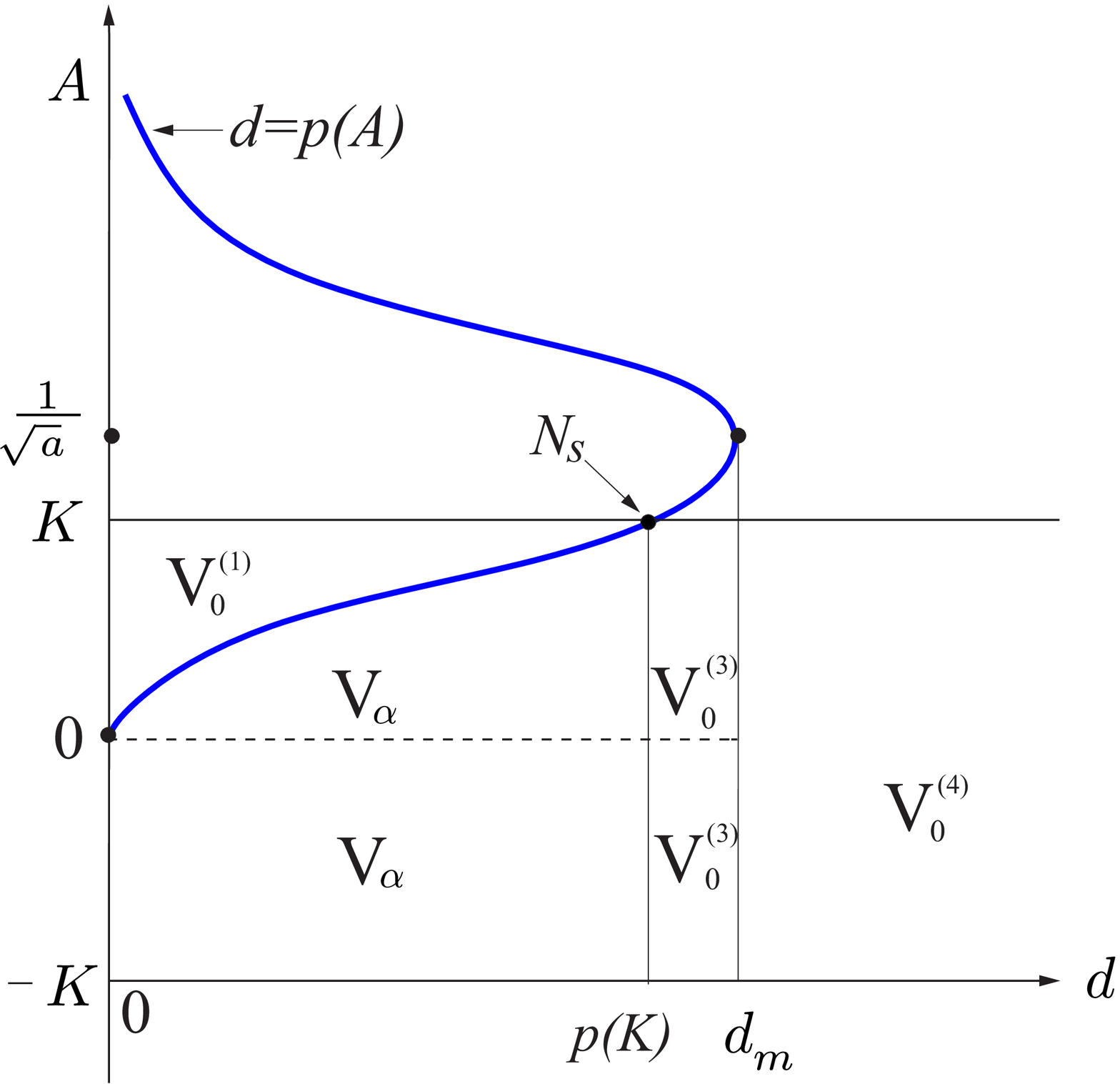}}
\caption{Existence and bifurcations of equilibrium points in $dA$-plane, where $E_\alpha$ and $E_\beta$ coexist in $V_{\alpha\beta}$. In $V_\alpha$ only $E_\alpha$ exists. In $V_\beta$ only $E_\beta$ exists. $E_\alpha$ and $E_\beta$ do not exist in $V_0^{(1)}\cup V_0^{(2)}\cup V_0^{(3)}\cup V_0^{(4)}$.}
\label{two_figure}
\end{center}\end{figure}

As $d$ increases, we obtain the following bifurcations involving $E_A$.
\begin{enumerate}
\item If $0<A<\frac{1}{\sqrt[]{a}}$, then a transcritical bifurcation involving $E_\alpha$ and $E_A$ occurs when $d=p(A)$. $E_A$ changes its stability from an unstable node to a saddle point.
\item If $A>\frac{1}{\sqrt[]{a}}$, then a transcritical bifurcation involving $E_\beta$ and $E_A$ occurs when $d=p(A)$. $E_A$ changes its stability from an unstable node to a saddle point.
\end{enumerate}

Also as $d$ increases, we obtain the following bifurcations involving $E_K$.
\begin{enumerate}
\item If $K<\frac{1}{\sqrt[]{a}}$, then a transcritical bifurcation involving $E_\alpha$ and $E_K$ occur when $d=p(K)$. $E_K$ changes its stability from a saddle point to a stable node.
\item If $K>\frac{1}{\sqrt[]{a}}$, then a transcritical bifurcation involving $E_\beta$ and $E_K$ occur when $d=p(K)$. $E_K$ changes its stability from a saddle point to a stable node.  A saddle-node bifurcation involving $E_\alpha$ and $ E_\beta$ occurs when $d=d_m$.
\item If $K=\frac{1}{\sqrt[]{a}}$, then $E_\alpha,\ E_\beta$ and $E_K$ coalesce at $E_K$ when $d=d_m$.
\end{enumerate}

\begin{remark} ``$N_e$'' and ``$N_s$'' in Fig. \ref{two_figure} stand for nilpotent elliptic point and nilpotent saddle, respectively. Local bifurcations near nilpotent elliptic point and nilpotent saddle are studied in Section 5.1.
\end{remark}


\section{Bogdanov-Takens bifurcation: cusp of order 3}

If system \eqref{third} has an interior equilibrium $(x, y)$ as a nilpotent singularity, the variational matrix \eqref{vm} should have double zero eigenvalues. This happens when $p'(x)=0$ and $G'(x)=0.$ From the first condition we have $\alpha=\frac{1}{\sqrt{a}}=\beta$ and
\begin{equation*}
     E_\alpha= E_\beta=\left(\frac{1}{\sqrt[]{a}},G\left(\frac{1}{\sqrt[]{a}}\right)\right):=\bar{E}.
\end{equation*}
Also from section 2.1, when $E_\alpha = E_\beta$ we have
\begin{equation*}
    d=\frac{1}{b+2 \ \sqrt[]{a}}= d_m.
\end{equation*}
From the second condition we have
\begin{equation*}
    G'(\alpha)= G'\left(\frac{1}{\sqrt[]{a}}\right)=\frac{[(Ka-2\sqrt{a})A-2K\sqrt{a}+3](2\ \sqrt[]{a}+b)}{a}=0.
\end{equation*}
Solve for $A$ and we obtain
\begin{equation} \label{Avalue}
    A=\frac{2K\sqrt[]{a}-3}{Ka-2\ \sqrt[]{a}}:= A^*, \ \  \left(K\neq \frac{2}{\sqrt[]{a}}\right).
\end{equation}
We need $-K<A^*<K$ and $A^*<\frac{1}{\sqrt[]{a}}$ so that $\bar{E}$ is in the positive cone, and this happens only if
 \begin{equation} \label{aK}
     \frac{1}{\sqrt[]{a}} < K < \frac{\sqrt[]{3}}{\sqrt[]{a}}.
 \end{equation}

Suppose that $A=A^*$ and $d=d_m$. Now we reduce system \eqref{third} to its normal form near $\bar{E}$. We first translate $\bar{E}$ to the origin by an affine map $x_1=x-\alpha$ and $y_1=-p(\alpha)(y-G(\alpha))$, 
then take the Taylor series about the origin, and obtain
\begin{equation}\label{taylor_series}
\left\{\begin{split}
   \dot{x}_1&=\displaystyle y_1+\frac{p(\alpha)G''(\alpha)}{2}x_1^2 + Q_{10}(x_1,y_1),    \\
    \dot{y}_1&=\displaystyle -\frac{p(\alpha) p''(\alpha)G(\alpha)}{2}x_1^2+ Q_{20}(x_1,y_1),
\end{split}\right.
\end{equation}
where $Q_{i0}=O(|(x_1,y_1)|^3)$ is $C^{\infty}$ in $(x_1,y_1)$ $(i=1,2)$.

Using the near identity transformation
\begin{equation*}
  u= x_1, \ \ \ v= y_1+\frac{p(\alpha)G''(\alpha)}{2}x_1^2 + Q_{10}(x_1,y_1),
\end{equation*}
we have
\begin{equation} \label{co2}
\left\{\begin{split}
    \dot{u}&=v,    \\
    \dot{v}&=\delta_1u^2+\delta_2uv+Q_{21}(u,v),
\end{split}\right.
\end{equation}
where $Q_{21}=O(|(u, v)|^3)$ is $C^{\infty}$ in $(u, v)$, and
\begin{equation*}
   \delta_1=-\frac{p(\alpha)G(\alpha)p''(\alpha)}{2}, \ \ \delta_2= p(\alpha)G''(\alpha).
\end{equation*}
Notice that $\delta_2 =0$ if and only if $G''(\alpha)=0$ (i.e., $x=\frac{1}{\sqrt{a}}$ is an inflection point of $G(x)$). Then solving the equation
\begin{equation*}
 G''(\alpha)= G''\left(\frac{1}{\sqrt[]{a}}\right)=\frac{2K^2a^{3/2}+(-6Kb+10) \ \sqrt[]{a}+2K^2ab-8Ka+6b}{Ka-2\ \sqrt[]{a}} =0.
\end{equation*}
for $b$ yields that
\begin{equation*}
    b=-\frac{\sqrt{a}(K^2a-4K\sqrt{a}+5)}{K^2a-3K\sqrt[]{a}+3}:=b^*.
\end{equation*}
Here $b^*>-2\ \sqrt[]{a}$ for any $a$ and $K$ with the property \eqref{aK}.
Thus, we have following theorem.

\begin{theorem}\label{thm_cusp2}
The equilibrium $\bar{E}$ is a cusp singularity of order 2 when $\displaystyle (d,A)=\left(d_m,A^*\right)$ for any $a$, $K$ satisfying condition \eqref{aK} and for any $b>-2 \ \sqrt[]{a}$ except at $ b= b^*$.
\end{theorem}

\begin{remark}  It is required that $\displaystyle K\neq\frac{2}{\sqrt{a}}$ in Eq. \eqref{Avalue}.  If $\displaystyle K=\frac{2}{\sqrt{a}}$, $\bar{E}$ is not a cusp singularity since
$$G'\left(\frac{1}{\sqrt{a}}\right)=\frac{b+2\ \sqrt{a}}{a}\neq 0.$$
\end{remark}

Note that when $b=b^*$ we have $\delta_2=0$. Hence, $\bar{E}$ is more degenerate. Here, we will show that the order of the cusp singularity $\bar{E}$ is exactly 3 when $b=b^*$.
\begin{theorem}\label{thm_cusp3}
The equilibrium $\bar{E}$ is a cusp singularity of order 3 if
\begin{equation} \label{point}
    (d,A,b)=\left(\frac{K^2a-3K\sqrt[]{a}+3}{\ \sqrt[]{a}(K \ \sqrt[]{a}-1)^2}, \frac{2K\sqrt{a}-3}{\sqrt{a}(K\sqrt{a}-2)}, -\frac{\sqrt{a}(K^2a-4K\sqrt{a}+5)}{K^2a-3K\sqrt[]{a}+3} \right):=(d_m,A^*,b^*),
\end{equation}
for any $a$ and $K$ satisfying condition \eqref{aK}.
\end{theorem}

It has been shown in \cite{DRS} that any system with a cusp singularity of order 3 is locally $C^{\infty}$ equivalent to system
\begin{equation}
\left\{\begin{split} \label{co3}
    \dot{x}&= y,    \\
    \dot{y}&= x^2+\zeta x^3y+Q(x,y),
\end{split}\right.
\end{equation}
where $\zeta\neq 0$ and $Q(x,y)= O(|x,y|^5)$. So we need show that there exists a $C^{\infty}$ diffeomorphism which changes system \eqref{third} to system \eqref{co3} near $\bar{E}$. To achieve that we need the following proposition.

\begin{proposition}\label{t1}
System
\begin{equation} \label{5.3}
\left\{\begin{split}
    \dot{x}&=y+k_{30}x^3+k_{21}x^2y+k_{40}x^4+k_{31}x^3y + R_{10}(x,y),    \\
    \dot{y}&=m_{20}x^2+m_{30}x^3+m_{21}x^2y+m_{40}x^4+m_{31}x^3y + R_{20}(x,y),
\end{split}\right.
\end{equation}
is $C^{\infty}$ equivalent to system \eqref{co3} with
 \begin{equation} \label{coeff}
     \zeta=\frac{4k_{40}m_{20}+m_{20}m_{31}-3k_{30}m_{30}-m_{30}m_{21}}{m_{20}^4},
 \end{equation}
where $R_{i0}=O(|(x,y)|^5)$ is $C^{\infty}$ in $(x,y)$ $(i=1,2)$.
\end{proposition}

\begin{proof} See Appendix I.
\end{proof}

 \noindent
 \textbf{Proof of Theorem \ref{thm_cusp3}.} When $(d, A, b)=(d_m, A^*, b^*)$, system \eqref{taylor_series} becomes
\begin{equation*}
\left\{\begin{split}
    \dot{x}_1&=\displaystyle y_1+\frac{p(\alpha)G'''(\alpha)}{6}x_1^3+\frac{p''(\alpha)}{2p(\alpha)}x_1^2y_1+\frac{p(\alpha)G^{(4)}(\alpha)}{24}x_1^4+\frac{p'''(\alpha)}{6p(\alpha)}x_1^3y_1 + Q_{12}(x_1,y_1),    \\
    \dot{y}_1&=\displaystyle -\frac{p(\alpha)p''(\alpha)G(\alpha)}{2}x_1^2-\frac{p(\alpha)p'''(\alpha)G(\alpha)}{6}x_1^3+\frac{p''(\alpha)}{2}x_1^2y_1-\frac{p(\alpha)p^{(4)}(\alpha)G(\alpha)}{24}x_1^4\\
    &\ \ \ \ \ +\frac{p'''(\alpha)}{6}x_1^3y_1+ Q_{22}(x_1,y_1),
\end{split}\right.
\end{equation*}
where $Q_{i2}=O(|(x_1,y_1)|^5)$ is $C^{\infty}$ in $(x_1,y_1)$ $(i=1,2)$. 
Then by proposition \ref{t1} the system above is $C^\infty$ equivalent to system \eqref{co3} with
 \begin{align*}
\zeta &=\frac{4}{3G'''(\alpha)p^2(\alpha)(p''(\alpha))^4}\left(G'''(\alpha)p'''(\alpha)-G^{(4)}(\alpha)p''(\alpha)\right)  \\
  &=-a^{7/2} \left(\frac{(K\sqrt[]{a}-2)^2(3(K\sqrt[]{a})^4-22(K\sqrt[]{a})^3+62(K\sqrt[]{a})^2-78(K\sqrt[]{a})+39)}{((K\sqrt[]{a})^2-3(K\sqrt[]{a})+3)^6(K\sqrt[]{a}-1)}\right).
\end{align*}
For the sake of simplicity, define $\ell=K\sqrt[]{a}$. Then by condition \eqref{aK}, we have $1<\ell<\sqrt[]{3}$, and
$$\zeta=-\frac{a^{7/2}(\ell-2)^2(3\ell^4-22\ell^3+62\ell^2-78\ell+39)}{(\ell^2-3\ell+3)^6(\ell-1)}<0,$$
for any $1<\ell<\sqrt{3}$. Hence, $\bar{E}$ is a cusp singularity of order 3. \qed

In the following we will take $(d, A, b)$ as bifurcation parameters and develop a universal unfolding for cusp singularity of order 3 near the point $(d_m, A^*, b^*)$.  Let \begin{equation} \label{bt3points}
    b=b^*+\epsilon_1, \ A=A^*+\epsilon_2, \ d=d_m+\epsilon_3.
\end{equation}
Then the perturbed system is
\begin{equation} \label{new}
\left\{\begin{split}
   \dot{x}&=\displaystyle x(x-A^*-\epsilon_2)(K-x)-\frac{xy}{ax^2+(b^*+\epsilon_1)x+1}, \\
   \dot{y}&=\displaystyle y\left(\frac{x}{ax^2+(b^*+\epsilon_1)x+1}-d_m-\epsilon_3\right).
\end{split}\right.
\end{equation}
\begin{theorem}\label{BT3_thm}
For $\epsilon=(\epsilon_1,\epsilon_2,\epsilon_3)$ sufficiently small, system \eqref{new} is $C^\infty$ equivalent to
\begin{equation}\label{normal}
\left\{\begin{split}
   \dot{x}&=y, \\
   \dot{y}&=\eta_1(\epsilon)+\eta_2(\epsilon)y+\eta_3(\epsilon)xy+x^2-x^3y+R(x,y,\epsilon),
\end{split}\right.
\end{equation}
with
\begin{equation}\label{rest}
R(x,y,\epsilon)=y^2O(|x,y|^2)+O(|x,y|^5)+O(\epsilon)(O(y^2)+O(|x,y|^3))+O(\epsilon^2)O(|x,y|),
\end{equation}
where \begin{equation}\label{parameter}
\frac{\partial(\eta_1,\eta_2,\eta_3)}{\partial(\epsilon_1,\epsilon_2,\epsilon_3)}_{\big\rvert\epsilon=(0,0,0)} \neq 0.\end{equation}
Therefore, system \eqref{new} is a universal unfolding of Bogdanov-Taken bifurcation of codimension 3.
\end{theorem}

\begin{proof}
We will show that there exist a sequence of $C^\infty$ diffeomorphisms which convert system \eqref{new} to system \eqref{normal} and condition \eqref{parameter} is fulfilled. 

{\it Step 1.} Bring the cusp singularity $\left(\frac{1}{\sqrt[]{a}},\frac{(\ell-1)^4}{a(2-\ell)(\ell^2-3\ell+3)}\right)$ to the origin. Apply the translate
$$x_1=x-\frac{1}{\sqrt[]{a}}, \ \ \ y_1=y-\frac{(\ell-1)^4}{a(2-\ell)(\ell^2-3\ell+3)},$$
take the Taylor series about the origin, and we obtain
\begin{equation}
\left\{\begin{split}
   \dot{x}_1&=[k_{00}+k_{10}x_1+k_{20}x_1^2+k_{30}x_1^3+k_{40}x_1^4]+y_1[k_{01}+k_{21}x_1^2+k_{31}x_1^3]+Q_1(x_1,y_1), \\
   \dot{y}_1&=[m_{00}+m_{20}x_1^2+m_{30}x_1^3+m_{40}x_1^4]+y_1[m_{01}+m_{21}x_1^2+m_{31}x_1^3]+Q_2(x_1,y_1),
\end{split}\right.
\end{equation}
where
\begin{eqnarray*}
\left.\begin{array}{l}
  k_{00}=\frac{(\ell^2-3\ell+3)}{a^2(2-\ell)}\epsilon_1-\frac{(\ell-1)}{a}\epsilon_2+O(\epsilon^2), \hskip 1.7cm m_{00}=-\frac{(\ell^2-3\ell+3)}{a^2(2-\ell)}\epsilon_1-\frac{(\ell-1)^4}{a(\ell^3-3\ell+3)(2-\ell)}\epsilon_3+O(\epsilon^2), \\
  k_{10} =\frac{(2-\ell)}{\sqrt[]{a}}\epsilon_2+O(\epsilon^2),  \hskip 3.9cm
  m_{01}=-\frac{(\ell^2-3\ell+3)^2}{a(\ell-1)^4}\epsilon_1-\epsilon_3+O(\epsilon^2),\\
  k_{01} =\frac{-(\ell^2-3\ell+3)}{(\ell-1)^2\sqrt[]{a}}+\frac{(\ell^2-3\ell+3)^2}{(\ell-1)^4a}\epsilon_1+O(\epsilon^2),  \hskip 1.0cm
  m_{20}=\frac{-(\ell^2-3l+3)}{\sqrt[]{a}(2-\ell)}+\frac{2(\ell^2+3\ell+3)^2}{a(\ell-1)^2(2-\ell)}\epsilon_1+O(\epsilon^2),\\
  k_{20} =-\frac{2(\ell^2-3\ell+3)^2}{(2-\ell)(\ell-1)^2a}\epsilon_1+\epsilon_2+O(\epsilon^2),  \hskip 1.81cm
  m_{30}=\frac{(\ell^2-3\ell+3)}{(2-\ell)}-\frac{2(\ell^2-3\ell+3)^2}{\sqrt[]{a}(\ell-1)^2(2-\ell)}\epsilon_1+O(\epsilon^2),\\
  k_{30} =-\frac{(\ell^2-4\ell+5)}{(2-\ell)}+\frac{2(\ell^2-3\ell+3)^2}{\sqrt[]{a}(\ell-1)^2(2-\ell)}\epsilon_1+O(\epsilon^2),  \hskip 0.59cm
  m_{21}=-\frac{\sqrt[]{a}(\ell^2-3\ell+3)^2}{(\ell-1)^4}+\frac{2(\ell^2-3\ell+3)^3}{(\ell-1)^6}\epsilon_1+O(\epsilon^2),\\
  k_{21} =\frac{\sqrt[]{a}(\ell^2-3\ell+3)^2}{(\ell-1)^4}-\frac{2(\ell^2-3\ell+3)^3}{(\ell-1)^6}\epsilon_1+O(\epsilon^2),  \hskip 0.65cm
  m_{31}=\frac{a(\ell^2-3\ell+3)^2}{(\ell-1)^4}-\frac{2\sqrt[]{a}(\ell^2-3\ell+3)^3}{(\ell-1)^6}\epsilon_1+O(\epsilon^2),\\
  k_{31} =-\frac{a(\ell^2-3\ell+3)^2}{(\ell-1)^4}+\frac{2\sqrt[]{a}(\ell^2-3\ell+3)^3}{(\ell-1)^6}\epsilon_1 +O(\epsilon^2), \hskip 0.2cm
  m_{40}=\frac{\sqrt[]{a}(\ell^2-3\ell+3)}{(\ell-1)^2}-\frac{(\ell^2-3\ell+3)^2(\ell^2-5\ell+7)}{(\ell-1)^4(2-\ell)}\epsilon_1+O(\epsilon^2),\\
  k_{40} =\frac{\sqrt[]{a}(\ell^2-3\ell+3)}{-(\ell-1)^2}+\frac{(\ell^2-3\ell+3)^2(\ell^2-5\ell+7)}{(\ell-1)^4(2-\ell)}\epsilon_1+O(\epsilon^2),
\end{array}\right.
\end{eqnarray*}
and $Q_i(x_1,y_1)= O(|(x_1,y_1)|^5)$ is $C^\infty$ in $(x_1, y_1)$ ($i=1, 2$).

{\it Step 2.}  Apply the transformation $x_2=x_1$ and $y_2=\dot{x_1}$, and we obtain
\begin{equation}\label{step2}
\left\{\begin{split}
   \dot{x_2}&=y_2, \\
   \dot{y_2}&=[a_{00}+a_{10}x_2+a_{20}x_2^2+a_{30}x_2^3+a_{40}x_2^4]+y_2[a_{01}+a_{11}x_2+a_{21}x_2^2+a_{31}x_2^3]\\
            &\hskip 0.3cm+y_2^2[a_{12}x_2+a_{22}x_2^2]  +R(x_2,y_2,\epsilon),
\end{split}\right.
\end{equation}
where coefficients $a_{ij}$ are given in Appendix II, and $R(x_2,y_2,\epsilon)$ has the same property as $R(x,y,\epsilon)$ in Eq. \eqref{rest}. From now on we will use $R(x_i,y_i,\epsilon)$ in each step to denote terms which have same property as $R(x,y,\epsilon)$, but they are different functions.

{\it Step 3.} Remove $x_2y_2^2$ and $x_2^2y_2^2$-terms. Under the near identity transformation
$$\displaystyle x_3=x_2-\frac{a_{12}}{6}x_2^3-\frac{a_{22}}{12}x_2^4+O(|x_2,y_2|^5), \ \ y_3=y_2-\frac{a_{12}}{2}x_2^2y_2-\frac{a_{22}}{3}x_2^3y_2+O(|x_2,y_2|^5),$$
we obtain
\begin{equation} \label{step3}
\left\{\begin{split}
   \dot{x_3}&=y_3, \\
   \dot{y_3}&=[b_{00}+b_{10}x_3+b_{20}x_3^2+b_{30}x_3^3+b_{40}x_3^4]+y_3[b_{01}+b_{11}x_3+b_{21}x_3^2+b_{31}x_3^3]+ R(x_3,y_3, \epsilon).
\end{split}\right.
\end{equation}
where coefficients $b_{ij}$ are given in Appendix II.

{\it Step 4.} Reduce $x_3^3$ and $x_3^4$-terms in $O(\epsilon)$. Let\begin{equation} \label{11}
\left(b_{20}x_3^2+b_{30}x_3^3+b_{40}x_3^4\right)dx_3=b_{20}x_4^2dx_4.
\end{equation}

Note that $b_{20}=\frac{(\ell^2-3\ell+3)^2}{a(2-\ell)(\ell-1)^2}+O(\epsilon)>0$ for small $\epsilon$. We solve the differential form \eqref{11}, and obtain
\begin{equation*}
    x_3=\phi(x_4):=x_4-\frac{b_{30}}{4b_{20}}x_4^2-\left(\frac{3b_{40}}{5b_{20}}-\frac{3b_{30}^2}{16b_{20}^2}\right)x_4^3+O(x_4^4).
\end{equation*}

The differential form of system \eqref{step3} is
\begin{equation}\label{differential_1}
    y_3dy_3=[W_1(x_3,y_3)+(b_{20}x_3^2+b_{30}x_3^3+b_{40}x_3^4)+R(x_3,y_3, \epsilon)]dx_3,
\end{equation}
where \begin{equation*}
    W_1(x_3,y_3)=[b_{00}+b_{10}x_3]+y_3[b_{01}+b_{11}x_3+b_{21}x_3^2+b_{31}x_3^3].
\end{equation*}
Let $x_4=\phi^{-1}(x_3)$ and $y_4=y_3$. Then the differential form \eqref{differential_1} becomes
\begin{equation*}
\left.\begin{split}
    y_4dy_4&=[W_1(\phi(x_4), y_4)+R(\phi(x_4),y_4, \epsilon)]\phi ' (x_4)dx_4+b_{20}x_4^2dx_4\\
           &=[W_2(x_4, y_4)+R(x_4, y_4, \epsilon)]dx_4,
\end{split}\right.
\end{equation*}
where $W_2(x_4, y_4)=c_{00}+c_{10}x_4+c_{20}x_4^2+c_{30}x_4^3+c_{40}x_4^4+y_4[c_{01}+c_{11}x_4+c_{21}x_4^2+c_{31}x_4^3],$
and coefficients $c_{ij}$ are given in Appendix II. In particular, $c_{10}, c_{11}, c_{30}, c_{40}=O(\epsilon)$.

Note the terms $c_{30}x_4^3$ and $c_{40}x_4^4$ are of $O(\epsilon)$, so the differential form above is equivalent to
\begin{equation}\label{step4}
\left\{\begin{split}
   \dot{x_4}&=y_4, \\
   \dot{y_4}&=c_{00}+c_{10}x_4+c_{20}x_4^2+y_4[c_{01}+c_{11}x_4+c_{21}x_4^2+c_{31}x_4^3]+R(x_4,y_4,\epsilon).
\end{split}\right.
\end{equation}

{\it Step 5.} Rescale the coefficient of $x_4^2$, and remove $x_4$-term.

Notice that $c_{20}=\frac{(\ell^2-3\ell+3)^2}{a(2-\ell)(\ell-1)^2}+O(\epsilon)> 0$ for small $\epsilon$. Let $$x_5=x_4+\frac{c_{10}}{2c_{20}}, \ y_5=\frac{y_4}{\sqrt{c_{20}}},  \ \ \  \textrm{and} \ \ \ \tau=\sqrt{c_{20}}t,$$
rewrite $\tau$ as $t$, and we have
\begin{equation*}
\left\{\begin{split}
   \dot{x_5}&=y_5, \\
   \dot{y_5}&=d_{00}+x_5^2+y_5[d_{01}+d_{11}x_5+d_{21}x_5^2+d_{31}x_5^3] + R(x_5,y_5,\epsilon),
\end{split}\right.
\end{equation*}
where
\begin{eqnarray*}
\left.\begin{array}{l}
d_{00}=\dfrac{c_{00}}{c_{20}}+O(\epsilon^2), \ \ d_{01}=\dfrac{c_{01}}{\sqrt{c_{20}}}+O(\epsilon^2),\\
d_{11}=\dfrac{3c_{31}c_{10}^2-4c_{10}c_{20}c_{21}+4c_{11}c_{20}^2}{4c_{20}^2\sqrt{c_{20}}},\ d_{21}=\dfrac{2c_{20}c_{21}-3c_{10}c_{31}}{2c_{20}^2\sqrt{c_{20}}},\ d_{31}=\dfrac{c_{31}}{\sqrt{c_{20}}}.
\end{array}\right.
\end{eqnarray*}

{\it Step 6.} Remove the $x_5^2y_5$-term. Let $$x_6=x_5, \ \ y_6=y_5-\frac{d_{21}}{3}y_5^2 \ \ \  \textrm{and} \ \ \ t=\left(1-\frac{d_{21}}{3}y_5\right)\tau ,$$ and rewrite $\tau$ as $t$, and we have
\begin{equation}
\left\{\begin{split}
   \dot{x_6}&=y_6, \\
   \dot{y_6}&=\xi_{00}+x_6^2+y_6[\xi_{01}+\xi_{11}x_6+\xi_{31}x_6^3y_6]\\
            &\hskip 0.3cm +\xi_{02}y_6^2+\xi_{12}x_6y_6^2+\xi_{03}y_6^3+\xi_{13}x_6y_6^3+\xi_{22}x_6^2y_6^2+\xi_{04}y_6^4+R(x_6,y_6,\epsilon),
\end{split}\right.
\end{equation}
where
\begin{eqnarray*}
\left.\begin{array}{l}
\xi_{00}= d_{00}, \ \xi_{01}=d_{01}-d_{00}d_{21}, \ \xi_{11}=d_{11}, \ \xi_{31}=d_{31},\\
\xi_{02}=-\frac{d_{21}(d_{00}d_{21}+6d_{01})}{9}=O(\epsilon), \ \xi_{12}=-\frac{2d_{21}d_{11}}{3}=O(\epsilon), \ \xi_{03}=-\frac{2d_{21}^2(d_{00}d_{21}+3d_{01})}{27}=O(\epsilon).
\end{array}\right.
\end{eqnarray*}
Here $\xi_{02},\xi_{12},\xi_{03}=O(\epsilon)$ since $d_{00},d_{01}=O(\epsilon)$. Hence, we can put $y_6^2,x_6y_6^2,y_6^3$-terms in $R(x_6,y_6,\epsilon)$. Terms $x_6^2y_6^2,x_6y_6^3,y_6^4$ are also included in $R(x_6,y_6,\epsilon)$ since $x_6^2y_6^2,x_6y_6^3,y_6^4 \in O(y_6^2|x_6,y_6|^2)$. Therefore, we obtain
\begin{equation}  \label{step7}
\left\{\begin{split}
   \dot{x_6}&=y_6, \\
   \dot{y_6}&=\xi_{00}+x_6^2+y_6[\xi_{01}+\xi_{11}x_6+\xi_{31}x_6^3]+R(x_6,y_6,\epsilon).
\end{split}\right.
\end{equation}

{\it Step 7.} Rescale the coefficient of $x_6^3y_6$.

The coefficient $\xi_{31}=\rho+O(\epsilon)$, where
\begin{eqnarray*}
\left.\begin{array}{l}
\rho=-\frac{(3\ell^4-22\ell^3+62\ell^2-78\ell+39)a^{\frac{1}{2}}}{(\ell-1)^2(2-\ell)\sigma}<0, \ \ \text{and}\ \ \sigma=\frac{(\ell^2-3\ell+3)}{(\ell-1)\sqrt{a(2-\ell)}}>0
\end{array}\right.
\end{eqnarray*}
for all $1<\ell<\sqrt{3}$. Then $\xi_{31}<0$ for small $\epsilon$. Let $$x_8=\xi_{31}^{\frac{2}{5}}x_7,\ \   y_8=-\xi_{31}^{\frac{3}{5}}y_7 , \ \ t=-\xi_{31}^{-\frac{1}{5}}\tau.$$
Then system \eqref{step7} is converted to the system \eqref{normal} with
\begin{eqnarray*}
\left.\begin{array}{l}
    \eta_{1}= \xi_{31}^{\frac{4}{5}}\xi_{00}=
\frac{\rho^{\frac{4}{5}}}{a^{\frac{3}{2}}}\epsilon_1+\frac{(\ell-1)^4\rho^{\frac{4}{5}}}{a^{\frac{1}{2}}(\ell^2-3\ell+3)^2}\epsilon_3+O(\epsilon^2), \\
\eta_{2}=-\xi_{31}^{\frac{1}{5}}\xi_{01}=
-\frac{3\rho^{\frac{1}{5}}(\ell^2-4\ell+5)}{a^{\frac{3}{2}}(2-\ell)\sigma}\epsilon_1-\frac{\rho^{\frac{1}{5}}(2-\ell)}{a^{\frac{1}{2}}\sigma}\epsilon_2-\frac{3\rho^{\frac{1}{5}}(\ell^2-4\ell+5)(\ell-1)^4}{a^{\frac{1}{2}}(\ell^2-3\ell+3)^2(2-\ell)\sigma }\epsilon_3+O(\epsilon^2),\\
\eta_{3}=-\xi_{31}^{-\frac{1}{5}}\xi_{11}=\frac{(\ell^4-6\ell^3+18\ell^2-30\ell+21)}{2a(2-\ell)(\ell-1)^2\sigma\rho^{\frac{1}{5}} }\epsilon_1-\frac{(5\ell-9)(2-\ell)}{2(\ell-1)\sigma \rho^{\frac{1}{5}}}\epsilon_2-\frac{3(\ell^2-4\ell+5)(\ell-1)^4}{2(2-\ell)(\ell^2-3\ell+3)^2\sigma\rho^{\frac{1}{5}}}\epsilon_3+O(\epsilon^2).
\end{array}\right.
\end{eqnarray*}
Thus finally we have
\begin{align*}
 \displaystyle \frac{\partial(\eta_1,\eta_2,\eta_3)}{\partial(\epsilon_1,\epsilon_2,\epsilon_3)}_{\big|\epsilon=(0,0,0)}&=\rho^{\frac{4}{5}}\rho^{\frac{1}{5}}\rho^{\frac{-1}{5}}\left|\begin{pmatrix}
  \frac{1}{a^{\frac{3}{2}}}     & 0 & \frac{(\ell-1)^4}{a^{\frac{1}{2}}(\ell^2-3\ell+3)^2}\\
  -\frac{3(\ell^2-4\ell+5)}{a^{\frac{3}{2}}(2-\ell)\sigma} & -\frac{(2-\ell)}{a^{\frac{1}{2}}\sigma} & -\frac{3(\ell^2-4\ell+5)(\ell-1)^4}{a^{\frac{1}{2}}(\ell^2-3\ell+3)^2(2-\ell)\sigma } \\
  \frac{(\ell^4-6\ell^3+18\ell^2-30\ell+21)}{2a(2-\ell)(\ell-1)^2\sigma} &  -\frac{(5\ell-9)(2-\ell)}{2(\ell-1)\sigma}  &  -\frac{3(\ell^2-4\ell+5)(\ell-1)^4}{2(2-\ell)(\ell^2-3\ell+3)^2\sigma } \\
\end{pmatrix}
\right|\\
&= \displaystyle -\frac{2(\ell-1)^2\rho^{\frac{4}{5}}}{a^2\sigma^2} \neq 0.
\end{align*}
Therefore, system \eqref{new} is a universal unfolding of Bogdanov-Takens bifurcation of codimension 3.
\end{proof}

System \eqref{new} has the same bifurcation set with respect to $(\epsilon_1,\epsilon_2,\epsilon_3)$ as system \eqref{normal} has
with respect to $(\eta_1,\eta_2,\eta_3)$, up to a homeomorphism in the parameter space. See Fig. \ref{BTco3}.

\begin{figure}[!htp]
\begin{center}
    {\includegraphics[angle=0,width=0.7\textwidth]{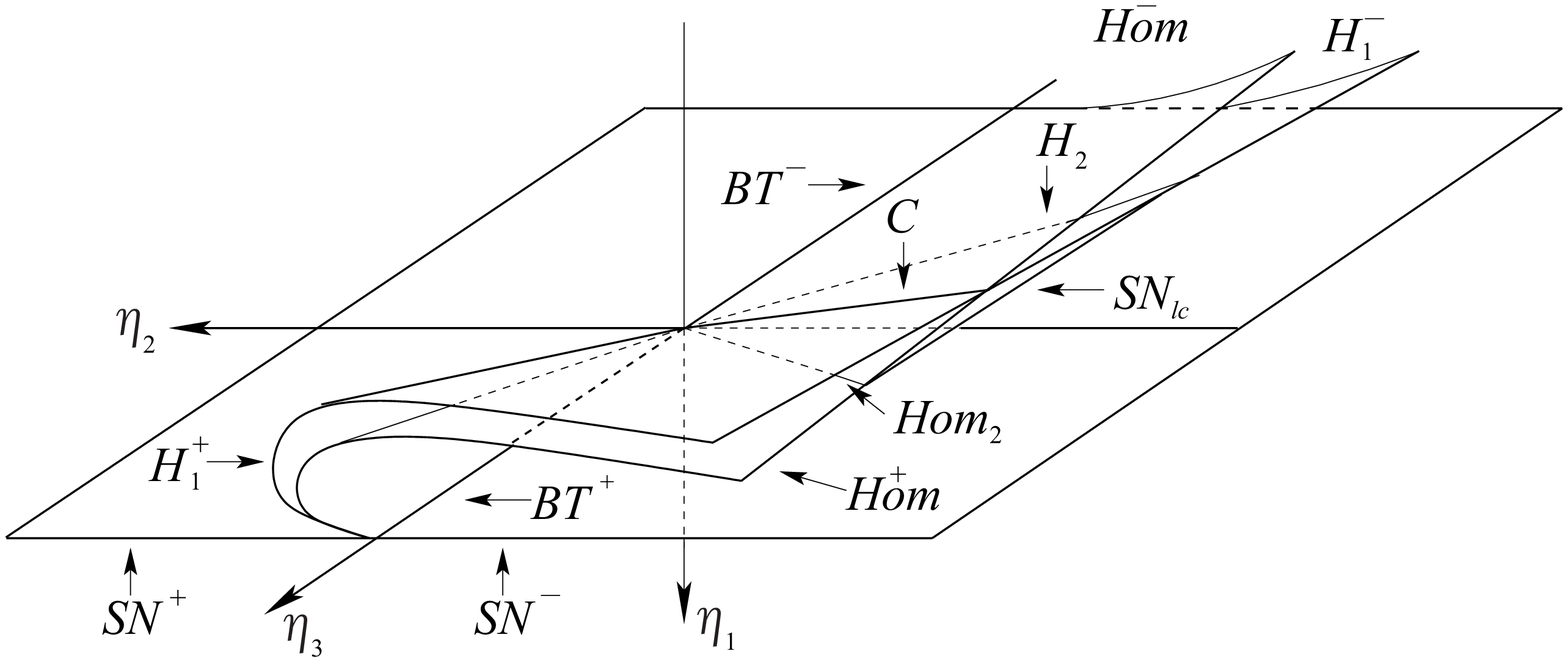}}
\caption{Bifurcation diagram of codimension 3 Bogdanov-Takens bifurcation.}
\label{BTco3}
\end{center}\end{figure}

For system \eqref{normal}, the plane $\eta_1=0$ except the origin in the parameter space is the surface of saddle-node
bifurcation. When $\eta_1>0$, system \eqref{normal} has no equilibria, and all bifurcation sets are in the half space $\eta_1<0$.
There are three bifurcation surfaces for $\eta_1<0$: a surface $H$ of Hopf bifurcation, a surface $Hom$ of
homoclinic bifurcation and a surface $SN_{lc}$ of saddle-node bifurcation of limit cycles, where surfaces $H$ and $Hom$ intersect at the curve $C$.
The surface $SN_{lc}$ is tangent to the surface $H$ at the curve $H_2$ where Hopf bifurcation of codimension 2 occurs. The surface $SN_{lc}$ is tangent to the surface $Hom$ at the curve $Hom_2$ where homoclinic bifurcation of codimension 2 occurs.
The surfaces $H$ and $Hom$ have first order tangency with the $\eta_1<0$ at $BT^+$ (i.e., positive $\eta_3$ axis) and at $BT^-$ (i.e., negative $\eta_3$-axis). System \eqref{normal} has a unique
unstable limit cycle between two surfaces $H_1^+$ and $Hom^+$ near the positive $\eta_3$-axis. There is a unique stable limit cycle between two surfaces $H_1^-$ and $Hom^-$ near the negative $\eta_3$-axis. If the parameter values are in the region bounded by three surfaces $H$, $H_{om}$ and $SN_{lc}$ near the origin,
system \eqref{normal} has exactly two hyperbolic limit cycles. The inner one is unstable and the outer one is stable.
These two limit cycles coalesce on surface $SN_{lc}$ where there exists a unique semi-stable limit cycle.

\begin{remark} When $\eta_1>0$ there is no equilibrium points for the system \eqref{third}. When $\eta_1<0$ we obtain two equilibrium points. Hence $\eta_1(\epsilon_1,\epsilon_3)=0$ or equivalently
\begin{equation*} \label{saddlenode}
    \epsilon_1=-\frac{a(\ell-1)^4}{(\ell^2-3\ell+3)^2}\epsilon_3+O(|\epsilon_3|^2)
\end{equation*}
corresponding to the saddle-node bifurcation surface in the parameter space $(\epsilon_1,\epsilon_2,\epsilon_3).$ Also by substituting \eqref{bt3points} to discriminant of \eqref{hx} we can find the exact saddle-node bifurcation surface given by
\begin{equation*}
    \epsilon_1=-\frac{a(\ell-1)^4\epsilon_3}{\sqrt{a}(\ell-1)^2(\ell^2-3\ell+3)\epsilon_3+(\ell^2-3\ell+3)^2}
\end{equation*}
which agrees with \eqref{saddlenode}. Similarly Hopf bifurcation curve for system \eqref{normal} is given by
\begin{equation*}
    \eta_2^2+\eta_1(\eta_1-\eta_3)^2=0,  \ \ \ \eta_1<0.
\end{equation*}
This is consistent with the expression of Hopf bifurcation surface given by substituting \eqref{bt3points} to $G'(\alpha)=0$.
\end{remark}

\begin{remark}\label{BT3A}
From Theorem 3.5, we know that Bogdanov-Takens bifurcation of codimension 3 occurs when $(d, A, b)=(d_m, A^*, b^*)$. By condition \eqref{aK}, $b^*$ is always negative.  However, $A^*$ could be either positive or negative depending on the values of $K$. In fact, we have 
$$A^*\in(-K, 0)\Longleftrightarrow K\in\left(\frac{1}{\sqrt{a}}, \frac{3}{2\ \sqrt{a}}\right), \ \textrm{and}\ \ A^*\in(0, K)\Longleftrightarrow K\in \left(\frac{3}{2\ \sqrt{a}}, \frac{3}{\sqrt{a}}\right).$$
Hence,  Bogdanov-Takens bifurcation of codimension 3 occurs when prey population has either weak or strong Allee effect. Furthermore, codimension 3 Bogdanov-Takens bifurcation and transcritical bifurcation between $E_{0}$ and $E_{A}$ occurs simultaneously at
\begin{equation*}
    (d,A,b,K)=\left(\frac{3}{\sqrt{a}},0,-\frac{5\ \sqrt{a}}{3},\frac{3}{2\ \sqrt{a}}\right).
\end{equation*}

\end{remark}


\section{Multiple periodic orbits}

\subsection{Existence of Hopf bifurcation}

From stability analysis in section 2, if there exists a periodic orbit, it contains $E_{\alpha}$ in its interior. The stability of $E_{\alpha}$ depends on the sign of $G'(\alpha)$.  Hence, Hopf bifurcation may occur as $G'(\alpha)$ changes its sign. From Eq. \eqref{td}  variational matrix $\mathbb{V}(\alpha,G(\alpha))$ has a pair of pure imaginary eigenvalues if $x=\alpha$ is a critical point of $G(x)$, equivalently, $x=\alpha$ is a root of
\begin{equation} \label{G'(x)}
    G'(x)=-4ax^3+3(aA+aK-b)x^2+2(aAK+bA+Kb-2)x+(A+K-AKb)=0.
\end{equation}

\begin{theorem}\label{critical}
The generic Hopf bifurcation occurs at $E_{\alpha}=(\alpha,G(\alpha))$ only if $G'(\alpha)=0$ and $G''(\alpha)\neq 0$.
\end{theorem}
\begin{proof}
We choose $d$ as a bifurcation parameter and examine conditions of Hopf bifurcation. Let $\lambda(d)$ and $\bar{\lambda}(d)$ be the two conjugated complex eigenvalues of $\mathbb{V}(\alpha, G(\alpha))$. If $G'(\alpha)=0,$ then $\operatorname{\mathbb{R}e(\lambda(d))}=\operatorname{\mathbb{R}e(\bar{\lambda}(d))}=0$. By Eq. \eqref{roots} and Eq. \eqref{td} we have
\begin{equation*}
   \frac{\partial (\operatorname{\mathbb{R}e}(\lambda(d)))}{\partial d}=\frac{\partial}{\partial \alpha}\left(\frac{1}{2}p(\alpha)G'(\alpha)\right)\cdot\frac{\partial\alpha}{\partial d} =\frac{p(\alpha)G''(\alpha)(a\alpha^2+b\alpha+1)}{2\sqrt{(bd-1)^2-4ad^2}}.
\end{equation*}
The inequality above holds as long as $G''(\alpha)\neq 0.$
\end{proof}

\subsection{Hopf Bifurcation of codimension 3}

There are several methods developed to study degenerate Hopf bifurcations, for instance, the method of Poincar\'e normal form \cite{GH,RS}, the method of averaging \cite{CH}, method of successive function\cite{ALGM}, the method of Lyapunov-Schmidt \cite{GL}, etc. These methods usually involve a fair amount of calculations, and direct application of these methods may lead to complicated focus quantities whose signs are difficult to determine. However, due to the special structure of predator-prey system, we are able to convert it to a generalized Li\'enard system
\begin{equation}  \label{Lienard}
\left\{\begin{split}
\displaystyle \frac{dx}{dt}&=\varphi(y)-F(x) ,\\
\displaystyle \frac{dy}{dt}&= -g(x),
\end{split}\right.
\end{equation}
where $xg(x)>0$ for $x\neq 0.$  Note that two variables $x$ and $y$ are separated and no cross terms appear on the right side of system \eqref{Lienard}. This feature is helpful to determine the order of a weak focus.

Now we introduce an important result about weak focus of the Li\'enard system \cite{han}.
\begin{lemma} \label{theorem62}
Suppose that $\varphi(y),\ F(x)$ and $g(x)$ are $C^{\infty}$ -functions in a neighborhood of the origin and
\begin{equation}\label{condition_Lienard}
    \varphi(0)=g(0)=F(0)=F'(0)=0,\ \varphi '(0)>0\ \ \textrm{and} \ \ g'(0)>0.
\end{equation}
Let $H(x)=\displaystyle \int_0^x g(s)ds.$ If there exists a $C^{\infty}$ function $\theta(x)=-x+O(x^2)$ such that $H(\theta(x))\equiv H(x)$ and
 \begin{equation*}
     F(\theta(x))-F(x)=\sum_{i\geq 1} B_ix^i,
 \end{equation*}
then the origin of system \eqref{Lienard} is a focus of order $k$ if $B_i=0$ for $i=1,2, ..., 2k$ and $B_{2k+1}\neq 0$. Furthermore, it is locally stable (resp., unstable) if $B_{2k+1}< 0$ (resp., $B_{2k+1}>0$).
\end{lemma}

By the Li\'enard transform, we change system \eqref{third} to Li\'enard system \eqref{Lienard} with
$$\varphi(y)=G(\alpha)(e^y-1), \ \ F(x)=G(-x+\alpha)-G(\alpha), \ \ g(x)=\frac{d}{p(-x+\alpha)}-1.$$
It is easy to verify $\varphi(y),\ F(x)$ and $g(x)$ satisfy condition \eqref{condition_Lienard}. As defined in Lemma \ref{theorem62}, we find 
\begin{equation*}
    H(x)=\int_0^x g(s)ds=-d\left(\frac{ax^2}{2}+\frac{x}{\alpha} +\ln\left(\frac{\alpha-x}{\alpha}\right)\right),
\end{equation*}
 and
 \begin{equation*}
     \theta(x)=-x+\mu_2x^2+\mu_3x^3+\mu_4x^4+\mu_5x^5+\mu_6x^6+O(x^7),
 \end{equation*}
where
\begin{equation}\label{Lienard_coeff}
\left.\begin{split}
    \mu_2 &=-\frac{2}{3\alpha(1-a\alpha^2)},  \\
    \mu_3 &=-\mu_2^2,\\
    \mu_4 &=\mu_2\left(2\mu_2^2+\frac{3\mu_2}{2\alpha}+\frac{3}{5\alpha^2}\right), \\
    \mu_5 &=-\mu_2^2\left(4\mu_2^2+\frac{9\mu_2}{2\alpha}+\frac{9}{5\alpha^2}\right),  \\
    \mu_6 &=\mu_2\left(9\mu_2^4+\frac{57\mu_2^3}{4\alpha}+ \frac{177\mu_2^2}{20\alpha^2}+\frac{12\mu_2}{5\alpha^3}+\frac{3}{7\alpha^4}\right).
\end{split}\right.
\end{equation}
such that $H(\theta(x))\equiv H(x)$. Here, we have $\mu_2<0$ and $\mu_3<0$.

\begin{theorem}\label{cod_Hopf}
The codimension of Hopf bifurcation is summarized as below.
\begin{table}[H]
\begin{center}
\begin{tabular}{c|c}
\hline Codimension   & \ \ \ \ \ \ \ \ \ \ \ \ \ \ \ \   Conditions \ \ \ \ \ \ \ \ \ \ \ \ \ \ \ \ \ \  \\
\hline
  $0$     & $G'(\alpha)\neq 0$ \\
  $1$     & $G'(\alpha)=0, G''(\alpha)\neq\dfrac{G'''(\alpha)}{3\mu_2}$  \\
  $2$     & $G'(\alpha)=0, G''(\alpha)=\dfrac{G'''(\alpha)}{3\mu_2}, G'''(\alpha)\neq\dfrac{40a\mu_2\alpha^2}{5\mu_2\alpha+2}$ \\
  $3$     & $G'(\alpha)=0, G''(\alpha)=\dfrac{G'''(\alpha)}{3\mu_2}, G'''(\alpha)=\dfrac{40a\mu_2\alpha^2}{5\mu_2\alpha+2}$ \\
\hline
\end{tabular}
\end{center}
\end{table}
\end{theorem}

\begin{proof}
The Taylor's series of $F(\theta(x))-F(x)$ about $x=0$ is
\begin{equation}\label{Fx}
\left.\begin{split}
    F(\theta(x))-F(x) &=G(-\theta(x)+\alpha)-G(-x+\alpha)=\displaystyle\sum_{i=1}^{7}B_ix^i+O(x^8),
\end{split}\right.
\end{equation}
where
\begin{align*}
    B_1=2G'(\alpha),\ \  \textrm{and}\ \  B_2=-\mu_2G'(\alpha).
\end{align*}
If $G'(\alpha)\neq 0$, $E_\alpha$ is a hyperbolic focus. If $G'(\alpha)=0$, we analyze the coefficients $B_i(i\geq3)$ to find the codimension of Hopf bifurcation. Indeed if $B_1=0$, we have
\begin{equation}\label{B3}
\left.\begin{split}
    B_3 &=\frac{1}{3}(G'''(\alpha)-3\mu_2G''(\alpha)) , \\
    B_4 &=-\frac{\mu_2}{2}\left(G'''(\alpha)-3\mu_2G''(\alpha)  \right).
\end{split}\right.
\end{equation}
If $G'''(\alpha)-3\mu_2G''(\alpha)\neq 0$, Hopf bifurcation is of codimension 1.  If $G'''(\alpha)-3\mu_2G''(\alpha)=0$, then both $B_3$ and $B_4$ are equal to zero. This corresponds to degenerate Hopf bifurcation, and system \eqref{third} has at least two limit cycles in the positive cone. Now we continue to analyze the higher order coefficients in Eq. \eqref{Fx}. By assuming $B_1=0$ and $B_3=0$ we find
\begin{equation}
\left.\begin{split}
    B_5 &=-\frac{1}{10\alpha^2}((5\mu_2\alpha+2)G'''(\alpha)-40a\mu_2\alpha^2),  \\
    B_6 &=\frac{\mu_2}{4\alpha^2}((5\mu_2\alpha+2)G'''(\alpha)-40a\mu_2\alpha^2).
\end{split}\right.
\end{equation}
It is clear that both $B_5$ and $B_6$ equal to zero if $(5\mu_2\alpha+2)G'''(\alpha)-40a\mu_2\alpha^2=0$. If we can find a set of parameters in the feasible region such that $B_5=0$,  $E_\alpha$ is a focus with order at least 3 which implies the codimension of Hopf bifurcation at $E_\alpha$ is at least 3. Now we assume $B_1=0, \ B_3=0$ and $B_5=0$ and compute
\begin{equation*}
    B_7=-\frac{96a(6a^2\alpha^4+23a\alpha^2+6)}{630\alpha^3(1-a\alpha^2)^2(2+3a\alpha^2)}<0.
\end{equation*}
Hence, the codimension of Hopf bifurcation is at most 3.

Now we will prove that degenerate Hopf bifurcation of codimension 3 exists by showing that there exists a set of parameters in the feasible region such that $B_1, B_3$ and $B_5$ vanish simultaneously and $B_7<0$. Inspired of expression of $B_i$'s, we introduce a diffeomorphism
\begin{equation*}
\left.\begin{split}
    \psi: \mathbb{R}^+\times (0, d_m)\times (-2\sqrt{a}, \infty)& \longrightarrow  \mathbb{R}^+\times (0, \frac{1}{\sqrt{a}})\times (-2\sqrt{a}, \infty), \\
         (K, d, b)&\longrightarrow  (K, \alpha, b),
\end{split}\right.
\end{equation*}
where $\alpha=\alpha(d, b)$ is defined in Eq. \eqref{hx}. Then we may consider $(K, \alpha, b)$ as new independent parameters instead of $(K, d, b)$.

Let $A=0$. Assuming $B_5=0$, we can solve for $b$ and obtain
\begin{equation*} \label{b}
    b(K,\alpha)=\frac{a(aK\alpha^2-36a\alpha^3+6K-44\alpha)}{3(a\alpha^2+2)}.
\end{equation*}
Substituting $b=b(K,\alpha)$ in expressions of $B_1$ and $B_3$ we obtain
\begin{equation}
\left.\begin{split}
    C_1(K,\alpha) &= \frac{(18K^2\alpha^3-72K\alpha^4+72\alpha^5)a^2+(12K^2\alpha-79K\alpha^2+90\alpha^3)a+6K-12\alpha}{9a\alpha^2+6}, \\
    C_2(K,\alpha) &= \frac{(18K^2\alpha^2-72K\alpha^3+48\alpha^4)a^2+(12K^2-88K\alpha+234\alpha^2)a-12}{9a\alpha^2+6}. \nonumber
\end{split}\right.
\end{equation}
Here we fix $a=0.00025573.$ Note that when $24.53545865<\alpha<24.53545867$, we have
\begin{align*}
    C_2(K,\alpha) &>0 \ \ \textrm{when} \ \ K=71.75583310, \\
    C_2(K,\alpha) &<0 \ \ \textrm{when} \ \ K=71.75583315.
\end{align*}
Similarly when $71.75583310<K<71.75583315$, we have
\begin{align*}
    C_1(K,\alpha) &>0 \ \ \textrm{when} \ \ \alpha=24.53545865, \\
     C_1(K,\alpha) &<0 \ \ \textrm{when} \ \ \alpha=24.53545867.
\end{align*}

\begin{figure}[H]
\begin{center}
    {\includegraphics[angle=0,width=0.45\textwidth]{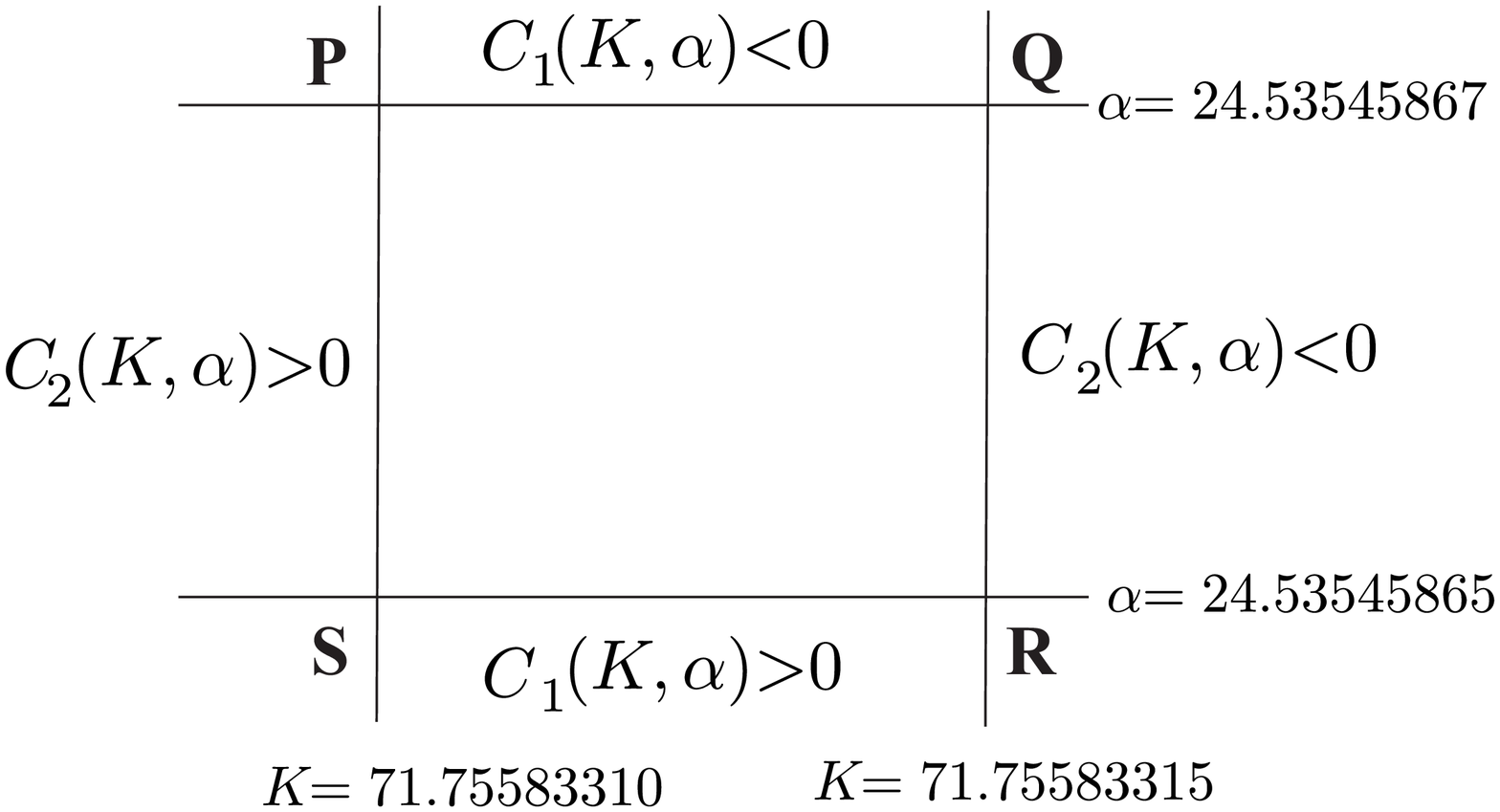}} \ \ \ \
    {\includegraphics[angle=0,width=0.42\textwidth]{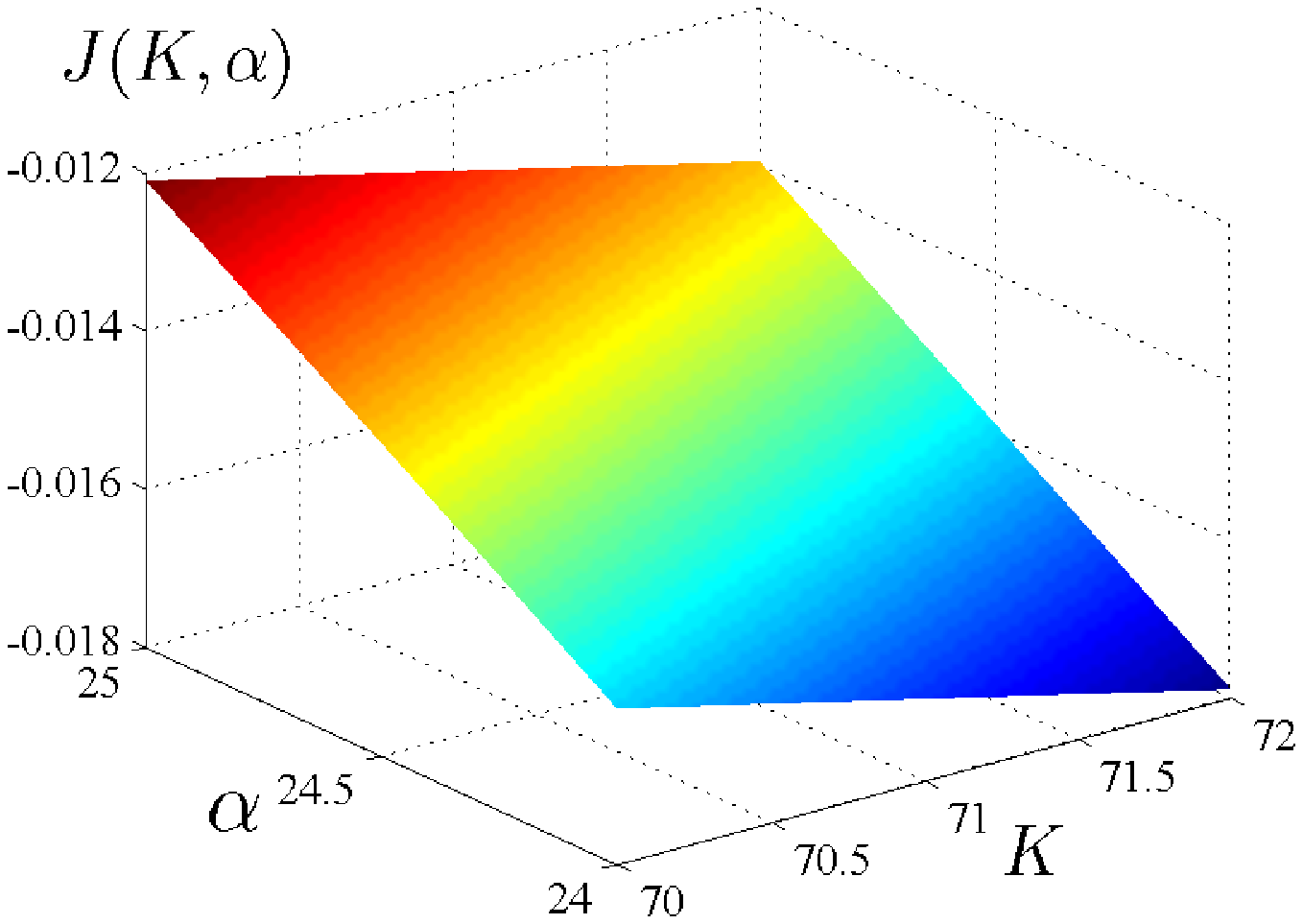}}
\caption{Rectangle {\bf PQRS} (left panel) and graph of $J=J(K, \alpha)$ (right panel).}
\label{NH3}
\end{center}\end{figure}

By Poincar\'e-Miranda theorem ~\cite{hopf} there exists $(K^*,\alpha^*)$ inside the rectangle \textbf{PQRS} such that $C_1(K^*,\alpha^*)=0$ and $C_2(K^*,\alpha^*)=0.$ It is easy to check that for any $(K,\alpha)$ inside the rectangle we have $b(K,\alpha)>-2\sqrt{a}=-0.03198312054$. Therefore, $b(K^*,\alpha^*)>-2\sqrt{a}$. Also note that for any $(K,\alpha)$ inside the rectangle \textbf{PQRS} we have
\begin{equation}\label{Jacobian}
    \displaystyle J(K,\alpha)= \frac{\partial(B_1, B_3, B_5)}{\partial(K, \alpha, b)}_{\big|b=b(K,\alpha)}<0,
\end{equation}
which can be seen from Fig. \ref{NH3}. Hence the non-degeneracy condition of Hopf bifurcation is fulfilled. Thus the codimension of the Hopf bifurcation is exactly 3.
\end{proof}

\begin{remark}
By inequality \eqref{Jacobian} and implicit function theorem, there exist smooth functions
$$K=\widetilde{K}(A),\ \ d=\tilde{d}(A),\ \  b=\tilde{b}(A)$$
in a small neighborhood of $A=0$ such that $B_i(\widetilde{K}(A), \tilde{d}(A), \tilde{b}(A), A)=0, i=1, 3, 5$.
Therefore,  Hopf bifurcation of codimension 3 can occur for both weak and strong Allee effect.
\end{remark}

From Theorem \ref{critical} Hopf bifurcation occurs at a critical point of $G(x)$. The following corollary describes the relation between the nature of critical point and the codimension of Hopf bifurcations.
\begin{corollary}
If $E_\alpha$ is an inflection point or a local minimum of $G(x)$, then the codimension of Hopf bifurcation is at most two. The codimension three Hopf bifurcation is possible only if $E_\alpha$ is a local maximum of $G(x)$.
\end{corollary}

\begin{proof}
Note that if $B_5=0$, then $G'''(\alpha)>0$, and hence $G''(\alpha)$ should be less than zero for $B_3=0$. Thus the desired result follows by considering three cases $G''(\alpha)=0$, $G''(\alpha)>0$ and $G''(\alpha)<0$.
\end{proof}

From Theorem \ref{cod_Hopf}, there exists  $(K, d, b)=(\hat{K}, \hat{d}, \hat{b})$ in the feasible region such that $B_i=0 (i=1, \cdots, 6)$, and system \eqref{third} has the same bifurcation set with respect to $(K, d, b)$ near $(\hat{K}, \hat{d}, \hat{b})$ as system
\begin{equation}  \label{Hopfnormal}
\left\{\begin{split}
\displaystyle \dot{r}&=\eta_1r+\eta_2r^3+\eta_3r^5-r^7+O(r^9),\\
\displaystyle \dot{\theta}&=-1+O(r^2),
\end{split}\right.
\end{equation}
has with respect to $(\eta_1,\eta_2,\eta_3)$, up to a homeomorphism in the parameter space. See Fig. \ref{Hopfcod3}.

\begin{figure}[!htp]
\begin{center}
    \includegraphics[angle=0,width=0.82\textwidth]{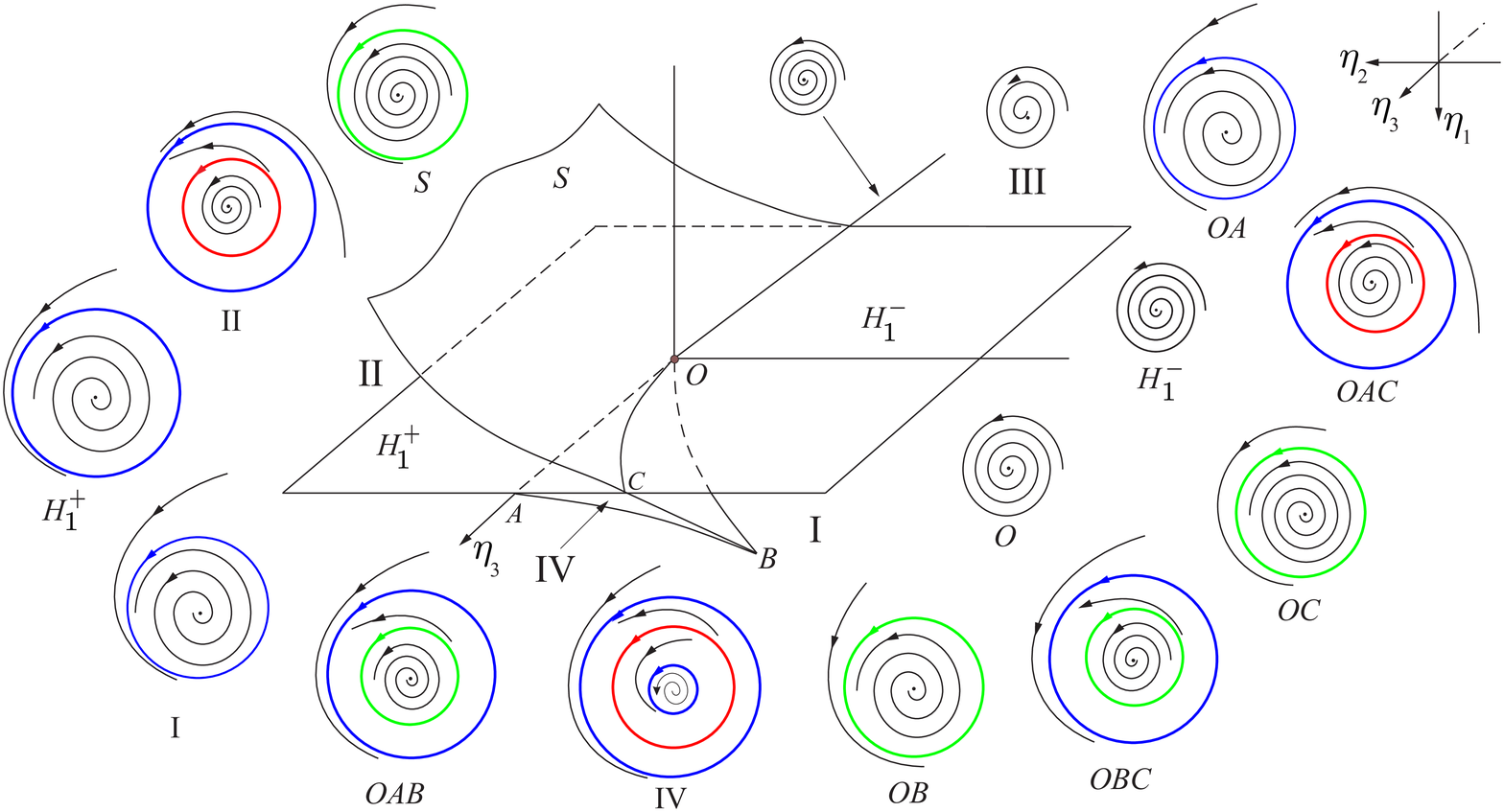}
\caption{Hopf bifurcation diagram of codimension 3.}
\label{Hopfcod3}
\end{center}\end{figure}

For system \eqref{Hopfnormal}, the supercritical Hopf bifurcation occurs on the half plane $H_1^-=\{(\eta_1, \eta_2, \eta_3)|\eta_1=0, \eta_2<0\}$, and subcritical Hopf bifurcation occurs on the half plane $H_1^+=\{(\eta_1, \eta_2, \eta_3)|\eta_1=0, \eta_2>0\}$. The Hopf bifurcation of co-dimension 2 occurs on the $\eta_3$-axis except the origin. The saddle-node bifurcation of limit cycles occurs on the surface $SN_{lc}$, which consists of three parts, the surface $S$ above the $\eta_2\eta_3$ plane, and surfaces $OAB$ and $OBC$ below the $\eta_2\eta_3$ plane. The $\eta_2\eta_3$ plane and the surface $SN_{lc}$ subdivide the parameter space into four generic regions, I, II, III and IV. The phase portraits of system \eqref{Hopfnormal} in generic regions and bifurcation sets are sketched in Fig. \ref{Hopfcod3}, where blue red and green circles denote stable, unstable and semistable limit cycles, respectively.

\noindent {\bf Examples.} We provide two examples in which system \eqref{third} has two and three limit cycles, respectively.

\begin{figure}[!htp]
\begin{center}
\subfigure[{Two limit cycles}]
    {\includegraphics[angle=0,width=0.48\textwidth]{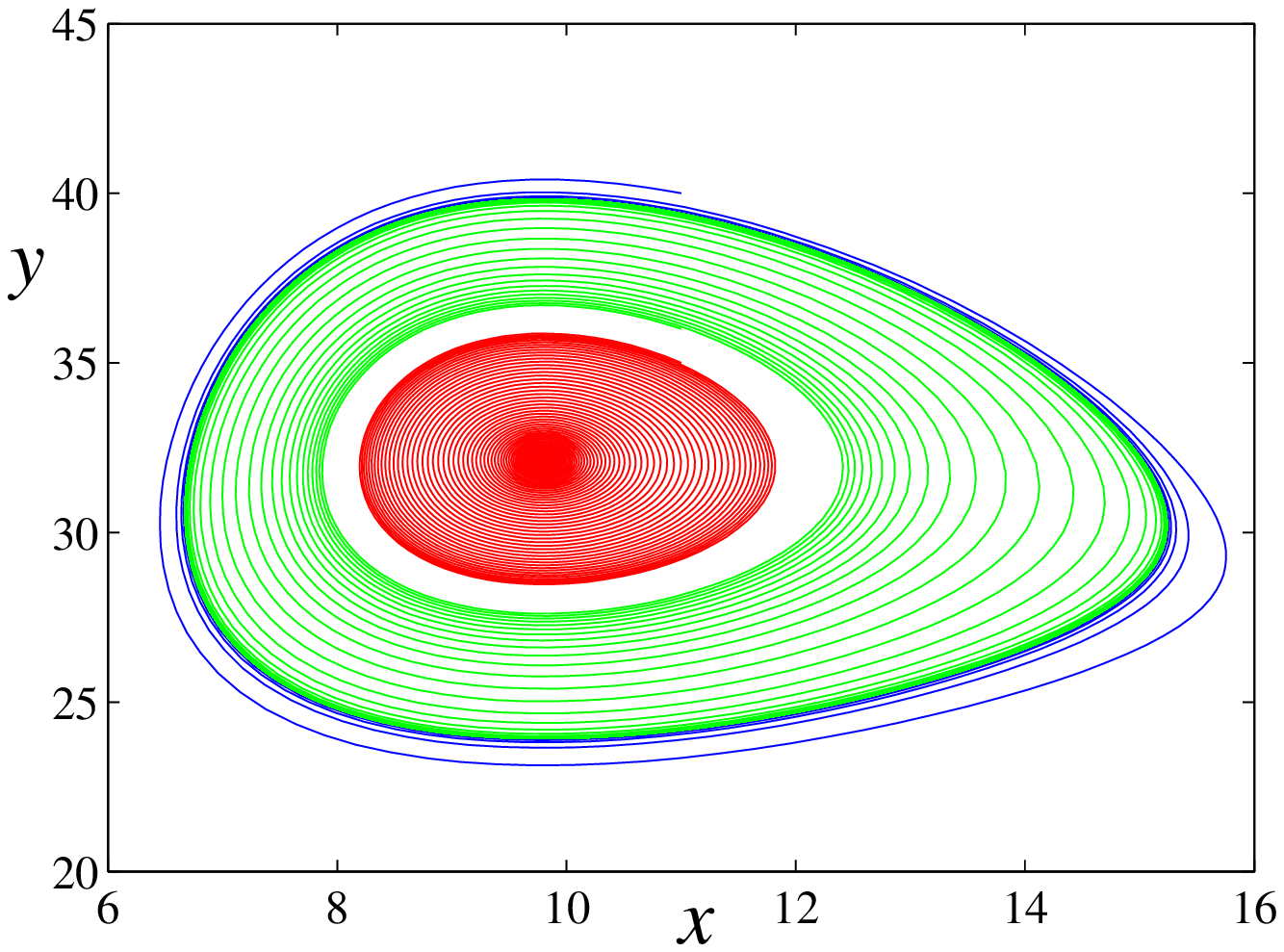}}\
\subfigure[{Three limit cycles}]
    {\includegraphics[angle=0,width=0.48\textwidth]{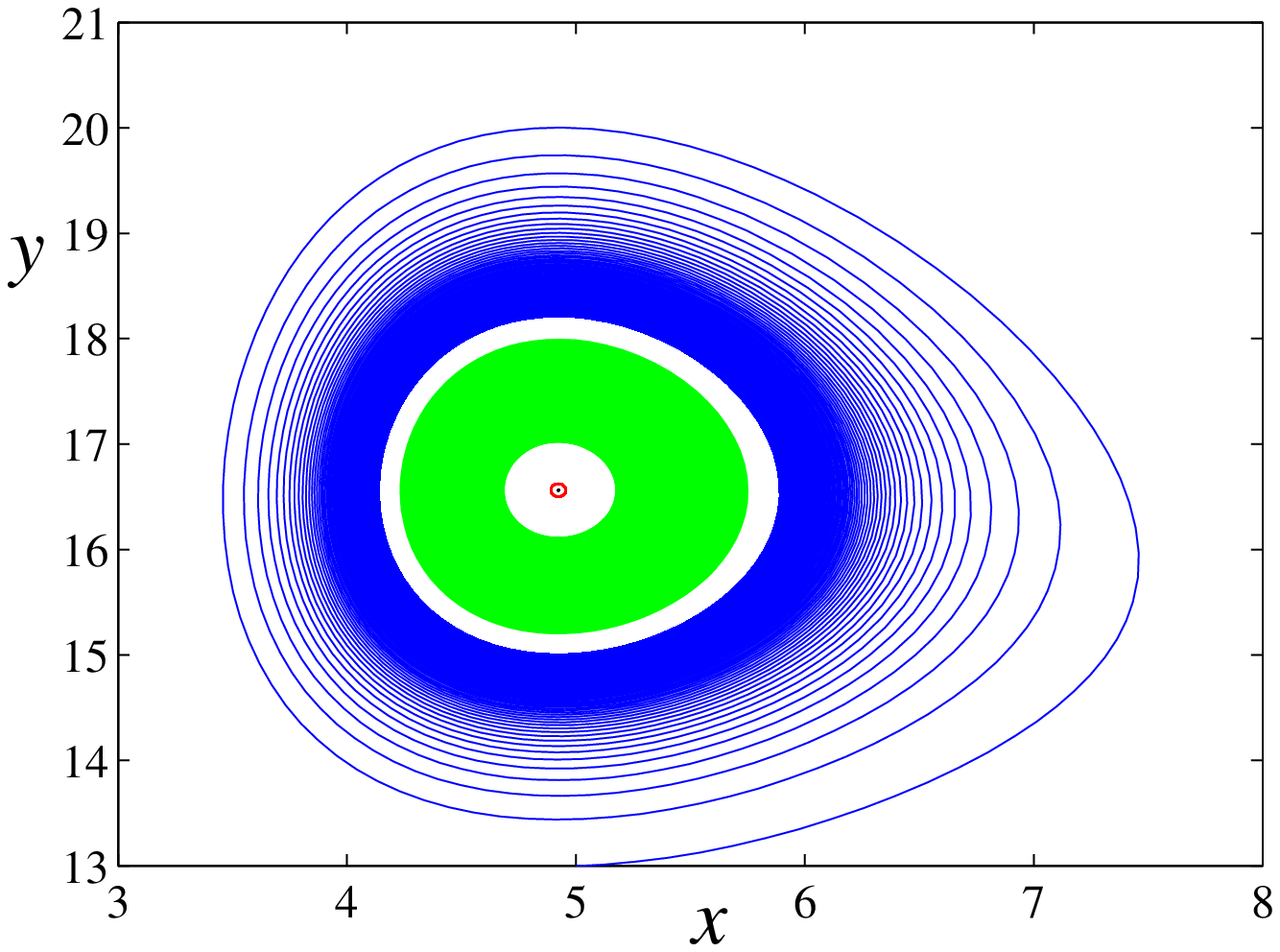}}
\caption{Multiple limit cycles.}
\label{po2}
\end{center}\end{figure}

Choose $K=20$, $A=2$, $a=0.004905$, $b=-0.10891$, and $d=24.28$. The orbit (blue curve) starting at $(11, 40)$ spirals inward, the orbit (green curve) starting at $(11, 36)$ spirals outward. The orbit (red curve) starting at $(11, 35)$ spirals inward and converges to $E_\alpha$. There are two limit cycles in this case. See Fig. \ref{po2} (a).

Choose $K=10.5$, $A=-0.5$, $a=0.01809954751$, $b=-0.1809954751,$ and $d=8.99$. The orbit (blue curve) starting at $(5, 13)$ spirals inward, the orbit (green curve) starting at $(5, 17)$ spirals outward. The orbit (red curve) starting at $(4.92, 16.5)$ spirals inward. However, $E_\alpha$ (black point) is unstable (i.e.,$G'(\alpha)>0$). There are three limit cycles in this case. See Fig. \ref{po2} (b).


\subsection{Existence and non-existence of periodic orbits}

In this section we explore sufficient conditions for existence and non-existence of periodic orbits of system \eqref{third}. The following theorem showed the position of periodic orbits whenever exist.

\begin{theorem}
If system \eqref{third} has a closed orbit,  then it lies entirely in the strip
\begin{equation}
    \{(x,y)\ |\ r_1\leq x\leq r_2\ \ \textrm{and}\ \ y>0  \}
\end{equation}
where $r_1=\max\{0,A\}$ and $r_2=\min\{\beta,K\}.$
\end{theorem}
\begin{proof}
It is apparent that  $r_1\geq 0$. Note that $\dot{x}_{|x=A}<0.$ So any orbit cannot cross the line $x=A$ from left to right. Thus if there exists a closed orbit, it should lie right to the line $x=A$. Similarly, we have $\dot{x}_{|x=K}<0$. Now if $r_2=\beta$, then $\dot{x}_{|x=\beta}<0$ (resp. $>0$) when $y>G(\beta)$ (resp. $y<G(\beta)$). Thus, any closed orbit crossing the line $x=\beta$ contains $E_{\beta}$ in its interior. However, by index theory it is impossible, because $E_{\beta}$ is a saddle point and $E_\alpha$ is a node, a focus or a cusp point. Hence, the desired result is proved.
\end{proof}

\begin{theorem}
System \eqref{third} has no closed orbits if either of the following condition holds: (a). $A \geq \frac{1}{\sqrt{a}};$\ \  (b). $A < \frac{1}{\sqrt{a}}$ \ \ and\ \  $0<d<p(A)$; \ \  (c). $G'(x)>0$ \ for all $x\in(r_1,\beta)$.
\end{theorem}
\begin{proof}
If condition $(a)$ or $(b)$ is satisfied, $E_{\alpha}$ is outside the positive cone, so closed orbits do not exist. For part $(c)$, we make the following change of variables
\begin{equation*}
    x_1=\ln{x}, \ \ y_1=\ln{y}, \ \ \textrm{and} \ \ \tau=\frac{t}{ax^2+bx+1},
\end{equation*}
and obtain
\begin{equation*}
\left.\begin{split}
\displaystyle \frac{dx_1}{d\tau}&=G(e^{x_1})-e^{y_1}, \ \ \displaystyle \frac{dy_1}{d\tau}=-h(e^{x_1}),
\end{split}\right.
\end{equation*}
where the function $h(\cdot)$ is defined by Eq. \eqref{hx}. Then if $G'(x)>0$ for all $x\in(r_1,\beta)$,
\begin{align*}
    \nabla\cdot\langle G(e^{x_1})-e^{y_1}, -h(e^{x_1}) \rangle = \frac{d}{dx_1}\left(G(e^{x_1})-e^{y_1}\right)+\frac{d}{dy_1}\left(-h(e^{x_1})\right)= e^{x_1}G'(e^{x_1})>0.
\end{align*}
By Dulac's criteria, system \eqref{third} has no closed orbits in the positive cone.
\end{proof}

\begin{theorem}
Suppose that $A\leq 0$ and $\alpha<K<\beta$. If $G'(\alpha)>0$,  system \eqref{third} has a periodic orbit in the positive cone containing $E_{\alpha}$.
\end{theorem}

\begin{proof} When the given condition holds, only $E_0,\ E_K$ and $E_{\alpha}$ exist, and stable manifolds of two saddles $E_0$ and $ E_K$ lie on the coordinate axes. Notice that all orbits initiating in the positive cone eventually enter the compact set bounded by $x$-axis, $y$-axis and the line $x+y=M$, and $E_{\alpha}$ is an unstable focus (or node). By Poincar\'e-Bendixson Theorem, there must be a periodic orbit in the positive cone containing $E_{\alpha}$.
\end{proof}


\section{Weak, strong Allee effect and biological interpretations}

As we know, Allee effect is weak when $-K<A<0$ and strong when $0<A<K$. A transcritical bifurcation between $E_{0}$ and $E_{A}$ occurs at $A=0$, and $E_A$ exists for $A>0$. The presence of $E_A$ will change the qualitative behaviors and dynamics of system \eqref{third}. In this section, we will first analyze the nilpotent saddle bifurcation and explore the distinct dynamics induced by strong Allee effects, then we explore the impact of both weak and strong Allee effects on the population dynamics of predators and preys. Biological significance of Allee effects are also discussed.

\subsection{Nilpotent saddle bifurcation}

Compared with weak Allee effects, strong Allee effects introduce complicated and distinct dynamics which are revealed by nilpotent saddle bifurcation below.
\begin{theorem}
Recall that $\ell=K\sqrt{a}$. If $A=K<\frac{1}{\sqrt{a}}$ and $d=p(A)$, then there exists a nilpotent saddle bifurcation in the neighborhood of the point $E_A$, which is a singular point of multiplicity 3. System \eqref{third} localized at $E_A$ is $C^\infty$ topologically equivalent to
\begin{equation} \label{saddle_2}
\left\{\begin{split}
    \dot{x}&=\displaystyle y+\gamma_1x^2,\\
    \dot{y}&=\displaystyle y(x+\gamma_2x^2+\gamma_3x^3+\gamma_4x^4+O(x^5)),
\end{split}\right.
\end{equation}
where \begin{equation} \label{gamma_12} \gamma_1=\frac{\ K(\ell^2+Kb+1)^2}{\ell^2-1}<0, \ \ \gamma_2=\frac{(\ell^2+Kb+1)[K(3\ell^2-5)b+6\ell^4-8\ell^2-2]}{2K(\ell^2-1)^2}.\end{equation}
If $\gamma_2\neq 0$, the bifurcation is codimension 2. $\gamma_2$ vanishes when $$b=b_{ns}:=\frac{2(3\ell^4-4\ell^2-1)}{K(5-3\ell^2)}.$$
If $b=b_{ns}$, then
\begin{equation} \label{gamma_34}\gamma_3=\frac{27(6\ell^4-13\ell^2+3)(\ell^2-1)^4}{K^2(3\ell^2-5)^4}, \ \ \ \gamma_4=\frac{81(21\ell^4-47\ell^2+6)(\ell^2-1)^6}{K^3(3\ell^2-5)^6}.\end{equation} The nilpotent saddle bifurcation is codimension 3 if $\ell\neq\left(\frac{47-\sqrt{1705}}{42}\right)^{1/2}$.
\end{theorem}

\begin{proof}
If $A=K<\frac{1}{\sqrt{a}}$ and $d=p(A)$, we have $\alpha=A=K$, and $E_\alpha$, $E_A$ and $E_K$ coalesce. So $E_A$ is a singular point of multiplicity 3.
By the affine map $(x, y)\rightarrow(x-\alpha, -p(\alpha)(y-G(\alpha)))$, we obtain
\begin{equation} \label{saddle_1}
\left\{\begin{split}
    \dot{x}&=\displaystyle y+\frac{p(\alpha)G''(\alpha)}{2}x^2+\frac{p'(\alpha)}{p(\alpha)}xy+\frac{p(\alpha)G'''(\alpha)+3p'(\alpha)G''(\alpha)}{6}x^3+\frac{p''(\alpha)}{2p(\alpha)}x^2y\\
    &\hskip 0.5cm +\frac{p(\alpha)G^{(4)}(\alpha)+4p'(\alpha)G'''(\alpha)+6p''(\alpha)G''(\alpha)}{24}x^4+\frac{p'''(\alpha)}{6p(\alpha)}x^3y + O(x^5, x^4y),    \\
    \dot{y}&=\displaystyle p'(\alpha)xy+\frac{1}{2}p''(\alpha)x^2y+\frac{1}{6}p'''(\alpha)x^3y+ yO(x^4).
\end{split}\right.
\end{equation}

There exists a change of variables and time preserving the invariance of the horizontal axis, which brings system \eqref{saddle_1} to system \eqref{saddle_2}, where
\begin{equation*}
\left.\begin{split}
     \gamma_1=\frac{p(\alpha)G''(\alpha)}{2p'(\alpha)}<0, \ \ \gamma_2=\frac{p(\alpha)p''(\alpha)G''(\alpha)-2(p'(\alpha))^2G''(\alpha)-p(\alpha)p'(\alpha)G'''(\alpha)}{2p(\alpha)(p'(\alpha))^2G''(\alpha)}.
\end{split}\right.
\end{equation*}
Expressions of $\gamma_3$ and $\gamma_4$ are complicated, so we will not present them here.

Note that $\ell=K\sqrt{a}$ and $0<\ell<1$, by which $\gamma_1$ and $\gamma_2$ are simplified as given in Eq. \eqref{gamma_12}. Straightforward calculation leads to the explicit conditions for the sign of $\gamma_2$. More specifically,

$(a)$ if $0<\ell\leq\frac{-3+\sqrt{21}}{6}$, then $\gamma_2<0$ for all $b>-2\sqrt{a}$;

$(b)$ if $\frac{-3+\sqrt{21}}{6}<\ell<1$, then $\gamma_2<0$ (resp. $\gamma_2>0$) if $b>b_{ns}$ (resp. $-2\sqrt{a}<b<b_{ns}$).


The nilpotent saddle bifurcation is of codimension 2 if $\gamma_2\neq 0$. If $\gamma_2=0$, we have $\frac{-3+\sqrt{21}}{6}<\ell<1$ and $b=b_{ns}$. Then expressions of $\gamma_3$ and $\gamma_4$ are significantly simplified by substituting $b=b_{ns}$, and they are given by Eq. \eqref{gamma_34}. Note that the coefficient $\gamma_4$ vanishes when $\ell=\left(\frac{47-\sqrt{1705}}{42}\right)^{1/2}\in \left(\frac{-3+\sqrt{21}}{6}, 1\right)$. Therefore, the nilpotent saddle bifurcation is codimension 3 if $\gamma_4\neq 0$.
\end{proof}

\begin{remark}
The codimension of nilpotent saddle bifurcation is one dimension less than that of ordinary nilpotent saddle bifurcation, because the saddle connection between $E_A$ and $E_K$ is fixed due to the invariance of $x$-axis.
\end{remark}

\begin{figure}[!htp]
\begin{center}
\includegraphics[angle=0,width=0.9\textwidth]{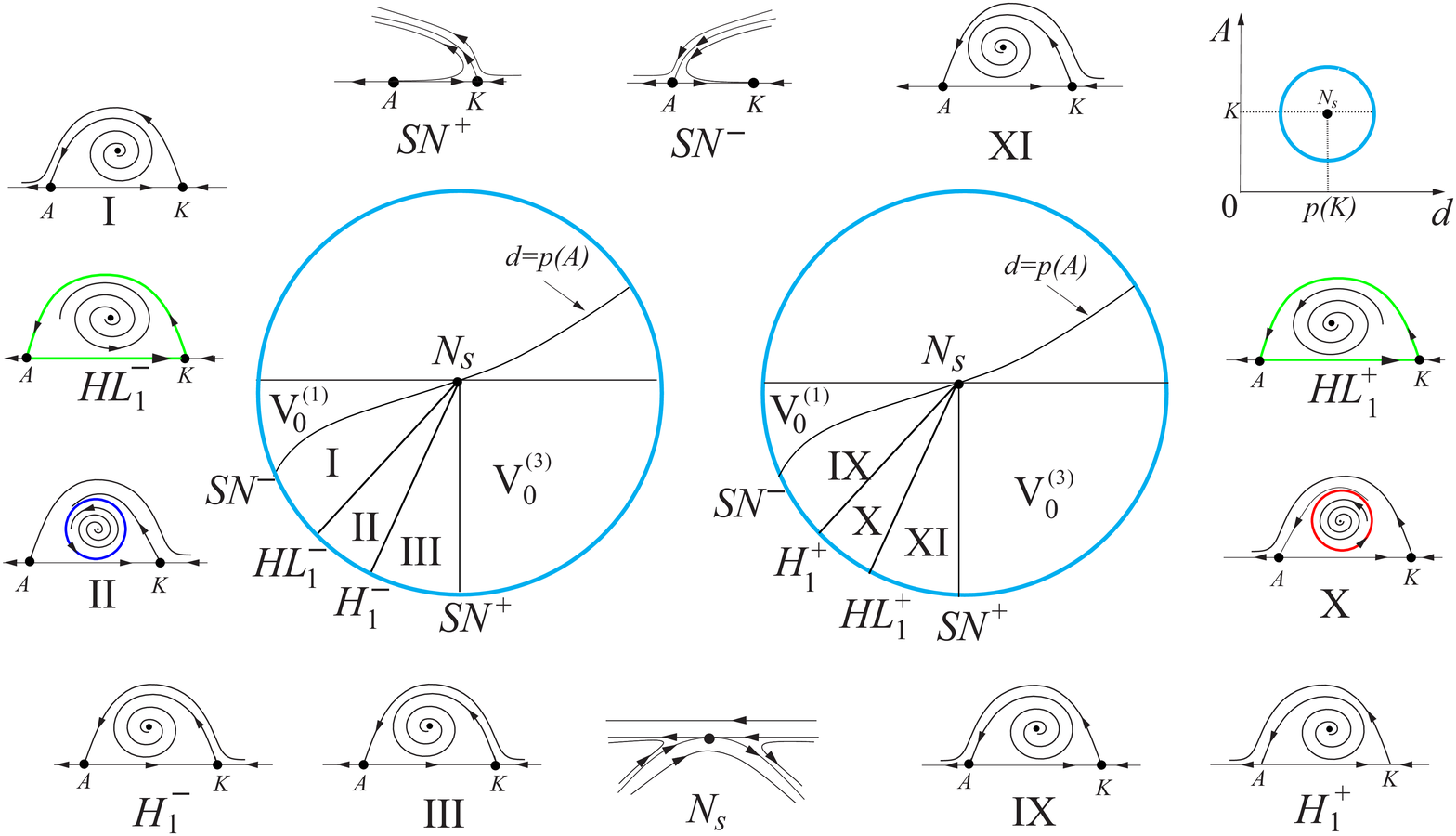}
\caption{Nilpotent saddle bifurcation diagram for $\gamma_2<0$ (left disk) and $\gamma_2>0$ (right disk).}
\label{ns2}
\end{center}\end{figure}

We sketch nilpotent saddle bifurcation diagram of codimension 2 in a small neighborhood of $(p(K), K)$ in $(d, A)$ plane. See Fig. \ref{ns2}, in which we plot the qualitative position of Hopf bifurcation curve and heteroclinic bifurcation curve. The feasible region is the lower part of the disk centered at $(p(K), K)$, which is subdivided in five regions by curves of saddle-node bifurcation, Hopf bifurcation and heteroclinic bifurcation. If $\gamma_2<0$, we have a curve of supercritical Hopf bifurcation $H_1^-$, and a curve of heteroclinic bifurcation $HL_1^-$. There exists a unique stable limit cycle between these curves $H_1^-$ and $HL_1^-$. If $\gamma_2>0$, we have a curve of subcritical Hopf bifurcation $H_1^+$, and a curve of heteroclinic bifurcation $HL_1^+$. There exists a unique unstable limit cycle between two curves $H_1^+$ and $HL_1^+$.  A transcritical bifurcation involving $E_\alpha$ and $E_K$ on $SN^+$. A transcritical bifurcation involving $E_\alpha$ and $E_A$ on $SN^-$. Phase portraits of system \eqref{third} on bifurcation curves and in subregions near the point $(p(K), K)$ are plotted in Fig. \ref{ns2}. Phase portraits for parameters $(d, K)$ in regions $V_0^{1}$ and $V_0^{1}$ are omitted because they are trivial.

\begin{remark}
If $A=K>\frac{1}{\sqrt{a}}$ and $d=p(A)$, then $E_\beta$, $E_A$ and $E_K$ coalesce, and $E_A$ is is a singular point of multiplicity 3. Note that $\ell>0$ in Eq. \eqref{gamma_12}, then $\gamma_1>0$, and $E_A$ is a nilpotent elliptic point. When one perturbs system \eqref{third} near the $(d, A)=(p(K), K)$, the perturbed system only has at most one singular point, i.e., the saddle $E_\beta$. The bifurcation diagram is just the part in a small neighborhood of the point $N_e$ in Fig. \ref{two_figure} (a).
\end{remark}



\subsection{Impact of  Allee effects and biological interpretations}

Allee effects have profound impact on dynamical behaviors of predator-prey systems, and we elaborate on their effects from the mathematical and biological points of view here.

The Allee effects induce richer and more complicated dynamics. For classical predator-prey models with Holling types functional response with logistic growth, Bogdanov-Takens bifurcation point of codimension 3 is the organizing center, and the codimension of degenerate Hopf bifurcation is two \cite{XZ, ZCW}. On the other hand, with either weak or strong Allee effects, the system \eqref{third} exhibits not only cusp type of Bogdanov-Takens bifurcation of codimension 3, but also the degenerate Hopf bifurcation of codimension 3. To the best of our knowledge, this is the first time that three limit cycles are discovered in predator-prey systems with Allee effects. Moreover, we have proven that nilpotent saddle/elliptic point bifurcation exists, a distinct phenomenon that is induced by strong Allee effects.

With the help of bifurcation analysis before, one can see the qualitative behaviors of model and impact of Allee effects by sketching phase portraits. Although it is impossible to sketch all the phase portraits here, the impact of Allee effects can be illustrated by three specific examples in Fig. \ref{connection3}. For simplicity, we denote $W^s(E)$ and $W^u(E)$ as the stable and unstable manifolds of a saddle point $E$. From Fig. \ref{connection3}, one can see that the basins of attraction are separated by separatrices, which are stable and unstable manifolds of saddle points $E_A$, $E_K$ and $E_\beta$. The final population sizes of two species depend on their initial populations.

\

\begin{figure}[!htp]
\begin{center}
\subfigure[{Homoclinic loop ($A<0$)}]
    {\includegraphics[angle=0,width=0.28\textwidth]{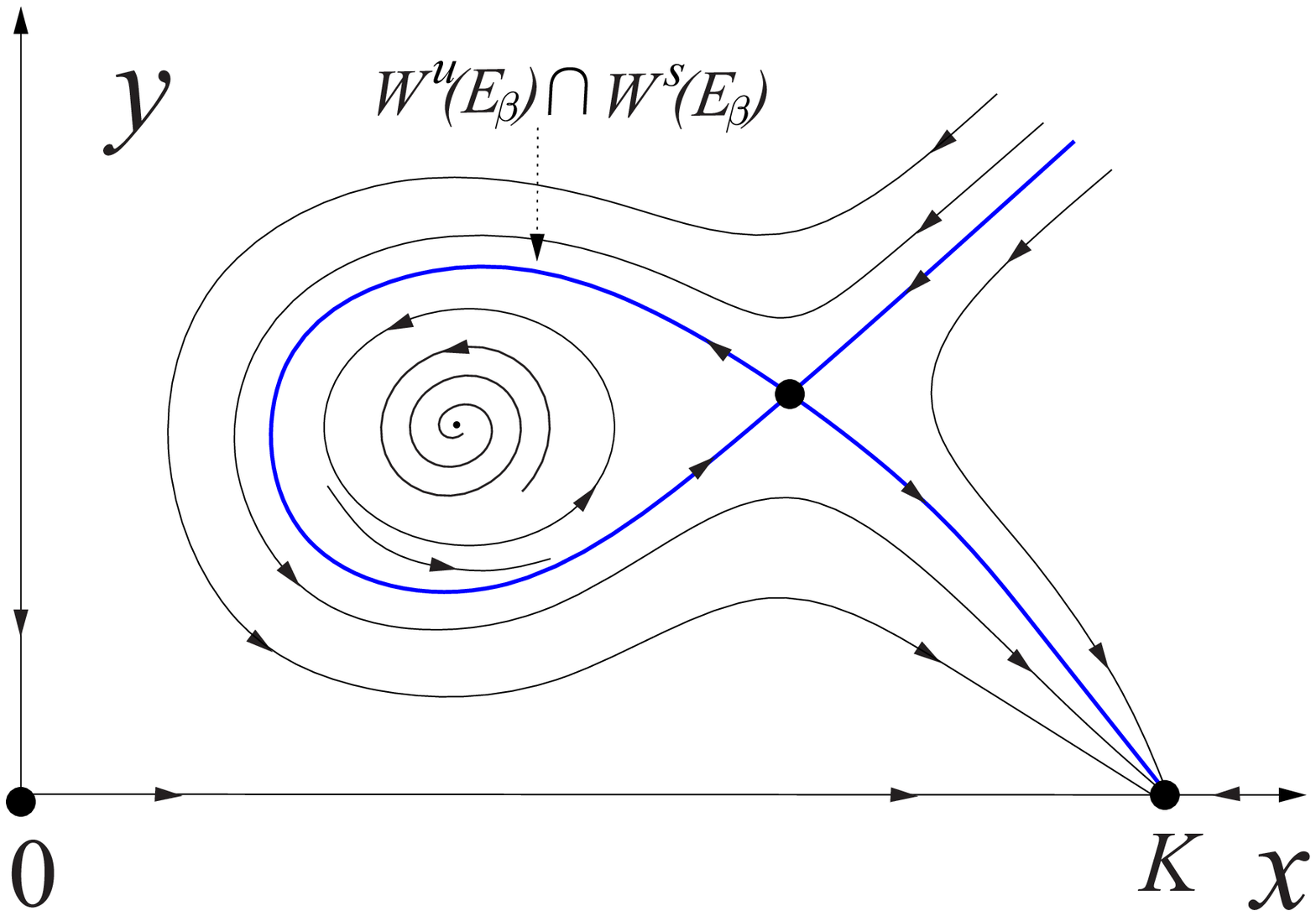}} \
\subfigure[{Homoclinic loop ($A>0$)}]
    {\includegraphics[angle=0,width=0.3\textwidth]{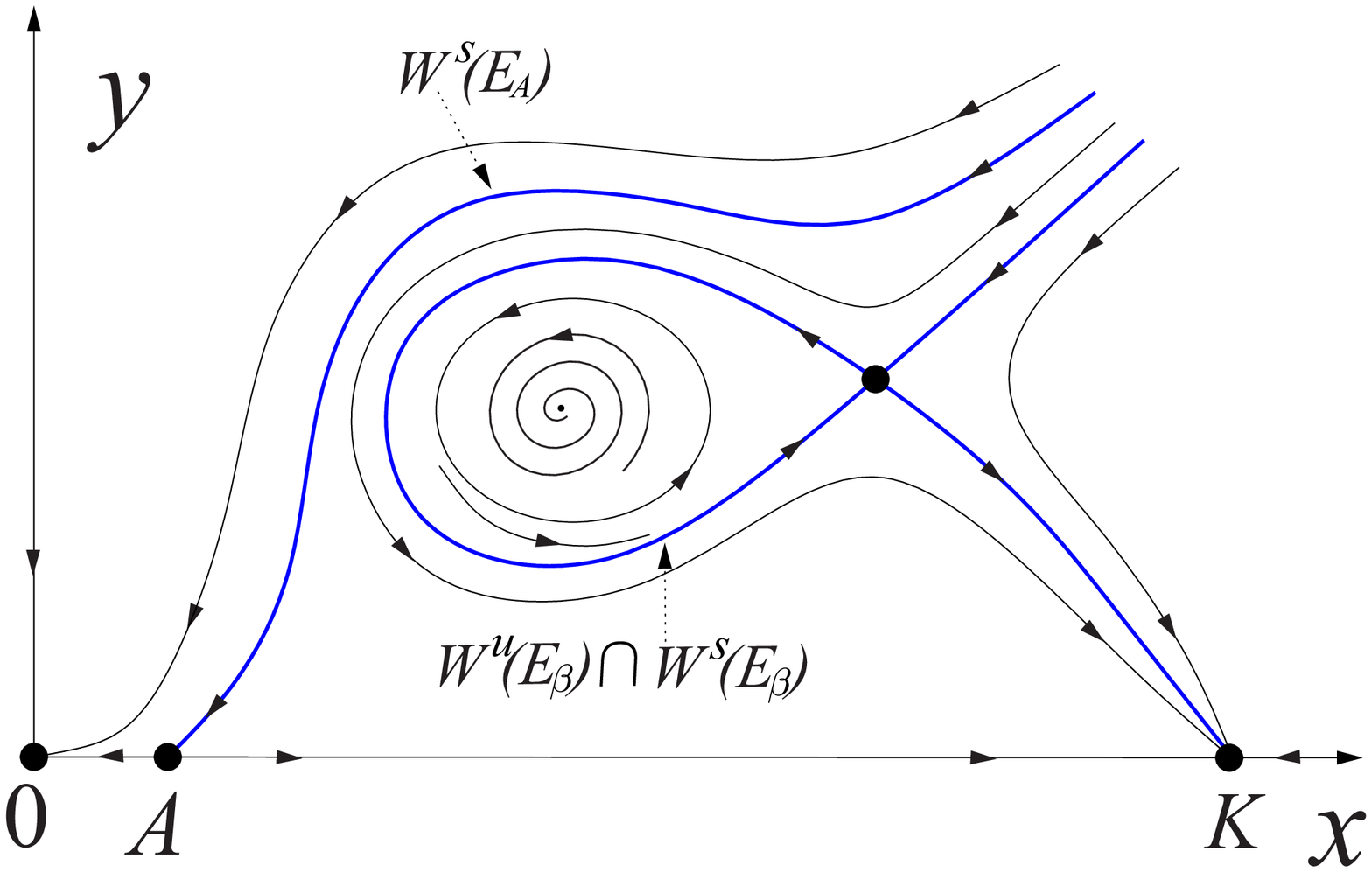}} \
\subfigure[{Heteroclinic loop ($A>0$)}]
    {\includegraphics[angle=0,width=0.3\textwidth]{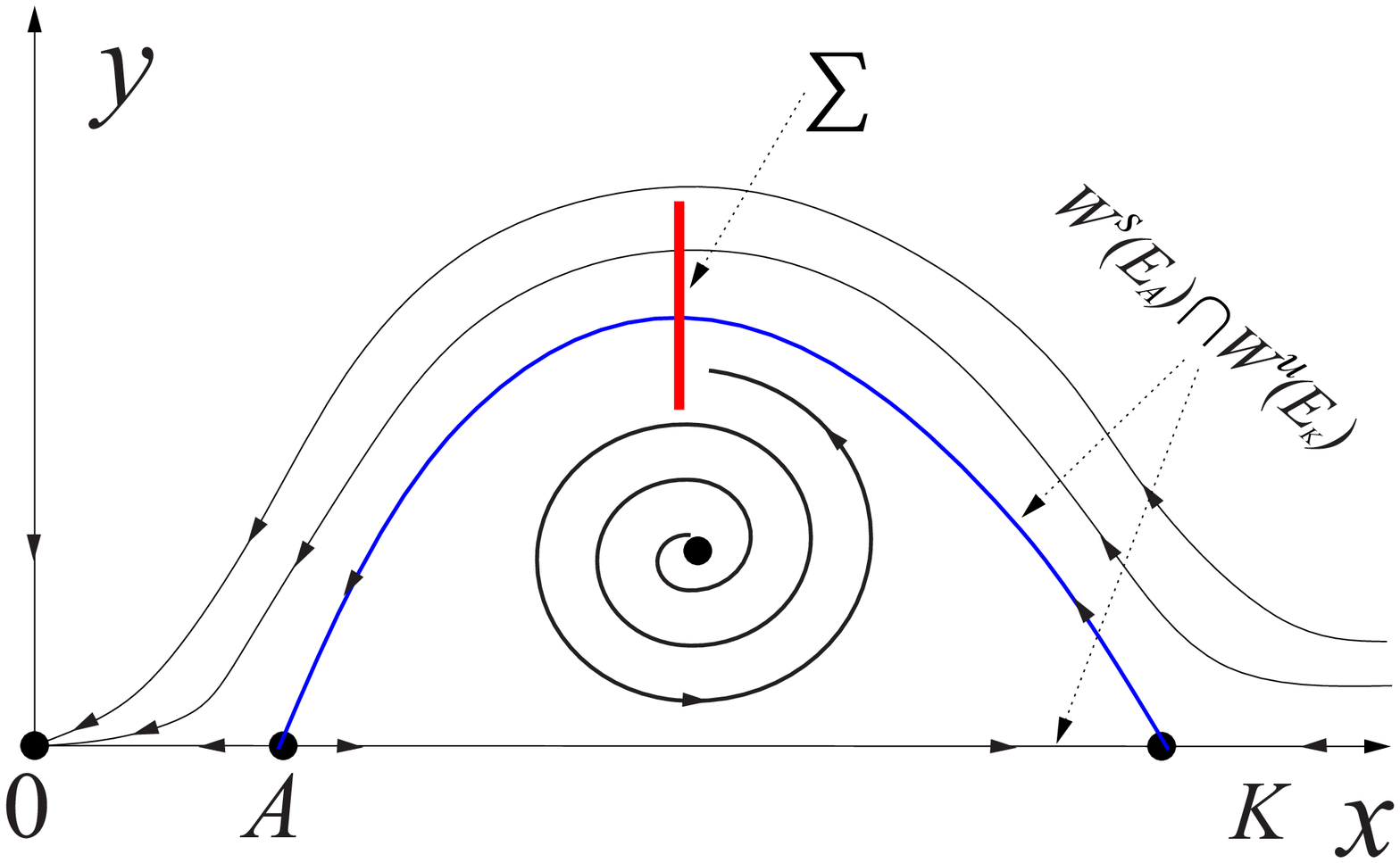}}
\caption{Codimension one phase portraits with Allee effects.}
\label{connection3}
\end{center}\end{figure}

Generally, basins of attraction of attractors are determined by the relative positions of stable and unstable manifolds of saddle points, which may change under perturbation. We take Fig. \ref{connection3} (c) as an example (i.e., $HL_1^-$ in Fig. \ref{ns2}), in which a heteroclinic loop exists, and preys and predators coexist for the initial condition (population) in the interior of the heteroclinic loop. We draw a line segment $\Sigma$ transversal to the heteroclinic orbit in the first quadrant.  After perturbation, if the point $W^u(E_K)\cap\Sigma$ is below the point $W^s(E_A)\cap\Sigma$, then a limit cycle emerges (see region {\bf II} in Fig. \ref{ns2}). Preys and predators coexist and their populations oscillate sustainably if the initial condition (population) is in the region beneath the stable manifold $W^s(E_A)$. Otherwise, both preys and predators are extinct. If the point $W^u(E_K)\cap\Sigma$ is above the point $W^s(E_A)\cap\Sigma$, then the first quadrant except $W^s(E_A)$ is the basin of attraction of $E_0$ (see region {\bf I} in Fig. \ref{ns2}), so both preys and predators die out!

The existence of homoclinic loops in Fig. \ref{connection3} (a) and (b) has been proven in Theorem \ref{BT3_thm} and by Remark \ref{BT3A}. Once homoclinic loops are broken under the perturbation, depending on the relative positions of stable and unstable manifolds of $E_\beta$, the basin of attraction of $E_K$ changes and a stable limit cycle which encloses the existing unstable limit cycle may emerge. One can easily sketch the phase portraits by Fig. \ref{BTco3}. Lastly, if comparing Fig. \ref{connection3} (a) with Fig. \ref{connection3} (b) and (c), one can see that the stable manifold $W^s(E_A)$ plays the role of a boundary, to its left both species are doomed to go extinct regardless of initial condition (population).

Compared with predator-prey models without Allee effects, Allee effects induce more steady states and sustained cycles. We summarize the distinct stability regimes that the Allee effect induce as below.

$(a)$ a stable equilibrium of coexistence and a stable equilibrium of extinction of both species;

$(b)$ up to two stable oscillations of populations and a stable equilibrium of extinction of both species;

$(c)$ a stable equilibrium of extinction of predators and a stable equilibrium of extinction of both species;

$(d)$ extinction of both species;

$(e)$ two stable oscillations of the populations.

Note that stability regimes $(a)$, $(b)$, $(c)$ and $(d)$ are induced by strong Allee effects, and stability regime $(d)$ is induced by weak Allee effects.

From the biological point of view, Allee effects in prey generally alter predator-prey dynamics in different ways. Firstly, strong Allee effects may introduce different scales of oscillations compared to systems without Allee effects. Secondly, strong Allee effects tend to destabilize coexistence of the two species. Strong Allee effects make populations more vulnerable to extinction due to the Allee threshold.  Sets of parameter values that predict stable coexistence in a model without Allee effects give rise to extinction of one or both species in the model with strong Allee effects. Thirdly, increasing the strength of strong Allee effects gradually reduces the range of initial conditions that lead to coexistence of both species. Finally, weak Allee effects in prey cause cycles between predators and preys for a wider range of model parameters than systems without Allee effects.

In summary, many species are confronted with Allee effects, either directly or through species they interact with. Allee effects should be taken into consideration in the management of populations of species, either for sustainable exploitation or for effective protection.




\section*{Appendix I: Proof of Proposition~\ref{t1}}
\setcounter{equation}{0}

\begin{proof}
  Using the near identity transformation
$$\displaystyle u=x-\frac{k_{21}}{3}x^3-\frac{k_{31}}{4}x^4+O(|(x,y)|^5),\ \  v=y+k_{30}x^3+k_{40}x^4+O(|(x,y)|^5),$$
\noindent
system \eqref{5.3} reduces to
\begin{equation*}
\left\{\begin{split}
    \dot{u}&=v,    \\
    \dot{v}&=\displaystyle m_{20}u^2+m_{30}u^3+(3k_{30}+m_{21})u^2v+\left(\frac{2m_{20}k_{21}+3m_{40}}{3}\right)u^4+(4k_{40}+m_{31})u^3v + R_{21}(u,v),
\end{split}\right.
\end{equation*}
where $R_{21}=O(|(u,v)|^5)$ is $C^{\infty}$ in $(u, v)$.

Let $u_2=m_{20}u_1$ and $v_2=m_{20}v_1$ and rewrite $u_2$ and $v_2$ as $x$ and $y$. Then we obtain
\begin{equation*}
\left\{\begin{split}
    \dot{x}&=y,    \\
    \dot{y}&=\displaystyle x^2+\left(\frac{m_{30}}{m_{20}^2}\right)x^3+\left(\frac{3k_{30}+m_{21}}{m_{20}^2}\right)x^2y+\left(\frac{2m_{20}k_{21}+3m_{40}}{3m_{20}^3}\right)x^4+\left(\frac{4k_{40}+m_{31}}{m_{20}^3}\right)x^3y + R_{22}(x,y),
\end{split}\right.
\end{equation*}
where $R_{22}=O(|(x, y)|^5)$ is $C^{\infty}$ in $(x, y)$. Then by proposition 5.3 of \cite{LCR}, the system above is $C^\infty$ equivalent to system \eqref{co3} with $\zeta$ given by \eqref{coeff}.
\end{proof}

\section*{Appendix II: Coefficients $a_{ij}, b_{ij}$ and $c_{ij}$ in Theorem~\ref{BT3_thm}}

Coefficients $a_{ij}$ in system \eqref{step2}:
\begin{eqnarray*}
\left.\begin{array}{l}
  a_{00}=-k_{00}m_{01}+m_{00}k_{01}, \quad a_{10}=k_{01}m_{10}-k_{10}m_{01}, \quad a_{01}=k_{10}+m_{01},\\
  a_{20}=k_{01}m_{20}+k_{21}m_{00}-k_{00}m_{21}-k_{20}m_{01},\\
  a_{30}=k_{01}m_{30}+k_{21}m_{10}+k_{31}m_{00}-k_{00}m_{31}-k_{10}m_{21}-k_{30}m_{01}, \\
  a_{40}=k_{01}m_{40}-k_{10}m_{31}-k_{20}m_{21}+k_{21}m_{20}+k_{31}m_{10}-k_{40}m_{01},\\
  a_{21}=\dfrac{3k_{01}k_{30}+k_{01}m_{21}-3k_{00}k_{31}-2k_{10}k_{21}}{k_{01}}, \quad a_{12}=\dfrac{2k_{21}^2}{k_{01}}, \quad a_{11}=\dfrac{2(k_{01}k_{20}-k_{00}k_{21})}{k_{01}},\\
  a_{31}=\dfrac{2k_{00}k_{21}^2+4k_{40}k_{01}^2+m_{31}k_{01}^2-3k_{01}k_{10}k_{31}-2k_{01}k_{20}k_{21}}{k_{01}^2}, \quad a_{22}=\dfrac{3k_{31}}{k_{01}}.
\end{array}\right.
\end{eqnarray*}
Coefficients $b_{ij}$ in system \eqref{step3}:
\begin{eqnarray*}
\left.\begin{array}{l}
b_{00}=a_{00}, \ b_{10}=a_{10}, \ b_{01}=a_{01}, \ b_{11}=a_{11},  \ b_{21}=a_{21},\\
b_{20}=a_{20}-\dfrac{a_{00}a_{12}}{2}, \ b_{30}=\dfrac{3a_{30}-a_{00}a_{22}-a_{10}a_{12}}{3},\\
b_{40}=\dfrac{12a_{40}-2a_{00}a_{12}^2-3a_{10}a_{22}-2a_{12}a_{20}}{12}, \ b_{31}=\dfrac{6a_{31}+a_{11}a_{12}}{6}.
\end{array}\right.
\end{eqnarray*}
Coefficients $c_{ij}$ in differential form \eqref{step4}:
\begin{eqnarray*}
\left.\begin{array}{l}
c_{00}=b_{00},\ c_{01}=b_{01}, \ c_{10}=\dfrac{2b_{10}b_{20}-b_{00}b_{30}}{2b_{20}},\\
c_{11}=\dfrac{2b_{11}b_{20}-b_{01}b_{30}}{2b_{20}},\ c_{20}=\dfrac{80b_{20}^3+45b_{00}b_{30}^2-48b_{00}b_{20}b_{40}-60b_{10}b_{20}b_{30}}{80b_{20}^2},\\
c_{21}=\dfrac{80b_{20}^2b_{21}+45b_{01}b_{30}^2-48b_{01}b_{20}b_{40}-60b_{11}b_{20}b_{30}}{80b_{20}^2}, \ c_{40}=\dfrac{b_{10}b_{30}(96b_{20}b_{40}-55b_{30}^2)}{48b_{20}^3},\\
c_{30}=\dfrac{(336b_{00}b_{20}b_{30}b_{40}-175b_{00}b_{30}^3-192b_{10}b_{20}^2b_{40}+210b_{10}b_{20}b_{30}^2)}{240b_{20}^3},\\
c_{31}=\dfrac{(336b_{01}b_{20}b_{30}b_{40}-175b_{01}b_{30}^3-192b_{11}b_{20}^2b_{40}+210b_{11}b_{20}b_{30}^2+240b_{20}^3b_{31}-240b_{20}^2b_{21}b_{30})}{240b_{30}^3}.
\end{array}\right.
\end{eqnarray*}


\end{document}